\documentclass[10pt,reqno]{amsart}

\usepackage{packages}
\usepackage{theorems}
\usepackage{notation}

\title[{%
  A Notion of TDIness for Convex, Semidefinite, and Extended Formulations%
}]{%
  A Notion of Total Dual Integrality for Convex,\\%
  Semidefinite, and Extended Formulations%
}

\author{Marcel K.\ de Carli Silva}
\address[Marcel K.\ de Carli Silva]{%
  Instituto de Matemática e Estatística, Universidade de São Paulo%
}
\email{mksilva@ime.usp.br}
\thanks{%
  Part of the results in this paper appeared in the PhD
  thesis~\cite{Carli13a} of the first author.\\
  \indent%
  Research of the first author was supported in part by a Sinclair
  Scholarship, a Tutte Scholarship, Discovery Grants from NSERC, and
  by U.S.~Office of Naval Research under award number
  N00014-12-1-0049, while at the Department of Combinatorics and
  Optimization, University of Waterloo, and by FAPESP
  (Proc.~2013/03447-6), CNPq (Proc.~477203/2012-4), CNPq
  (Proc.~456792/2014-7), and CAPES, while at the Institute of
  Mathematics and Statistics, University of São Paulo.%
}

\author{Levent Tunçel}
\address[Levent Tunçel]{%
  Department of Combinatorics and Optimization, University of Waterloo%
}
\email{ltuncel@uwaterloo.ca}
\thanks{%
  Research of the second author was supported in part by Discovery
  Grants from NSERC and by U.S.~Office of Naval Research under award
  numbers N00014-12-1-0049 and N00014-15-1-2171.  Part of this work was
  done while the second author was visiting the Simons Institute for
  the Theory of Computing, supported in part by the DIMACS/Simons
  Collaboration on Bridging Continuous and Discrete Optimization
  through NSF grant \#CCF-1740425.%
}

\date{January 27, 2018}

\begin{document}

\begin{abstract}
  Total dual integrality is a powerful and unifying concept in
  polyhedral combinatorics and integer programming that enables the
  refinement of \emph{geometric} min-max relations given by linear
  programming Strong~Duality into \emph{combinatorial} min-max
  theorems.  The definition of total dual integrality (TDI) revolves
  around the existence of optimal dual solutions that are integral,
  and thus naturally applies to a host of combinatorial optimization
  problems that are cast as integer programs whose LP~relaxations have
  the TDIness property.  However, when combinatorial problems are
  formulated using more general convex relaxations, such~as
  semidefinite programs (SDPs), it is not at all clear what an
  appropriate notion of integrality in the dual~program~is, thus
  inhibiting the generalization of the theory to more general forms of
  structured convex optimization.  (In~fact, we argue that the
  rank-one constraint usually added to SDP relaxations is not adequate
  in the dual SDP.)

  In this paper, we propose a notion of total dual integrality for
  SDPs that generalizes the notion for~LPs, by~relying on an
  ``integrality constraint'' for SDPs that is primal-dual symmetric.
  A key ingredient for the theory is a generalization to compact
  convex sets of a result of Hoffman for polytopes, fundamental for
  generalizing the polyhedral notion of total dual integrality
  introduced by Edmonds and Giles.  We~study the corresponding theory
  applied to SDP formulations for stable sets in graphs using the
  Lovász theta function and show that total dual integrality in this
  case corresponds to the underlying graph being perfect.  We~also
  relate dual integrality of an SDP formulation for the maximum cut
  problem to bipartite graphs.  Total dual integrality for extended
  formulations naturally comes into play in this context.
\end{abstract}

\maketitle

\section{Introduction}
\label{sec:intro}

In the polyhedral approach to combinatorial optimization one usually
starts by formulating a combinatorial problem as an integer linear
program (ILP) of the form
\(
\max\setst{
  \iprodt{c}{x}
}{
  Ax \leq b,\,
  x \geq 0,\,
  x \in  \Integers^n
}
\), which is relaxed into a linear program (LP) and then studied in
the light of LP duality.  This basic approach of polyhedral
combinatorics can be summarized by the following simple yet
fundamental result:
\begin{theorem}
  \label{thm:LP-chain}
  If \(A \in \Rationals^{m \times n}\) is a matrix, and \(b \in
  \Rationals^m\) and \(c \in \Rationals^n\) are vectors, then
  \begin{align}
    \label{eq:chain-ILP}
    &
    \sup\setst{
      \iprodt{c}{x}
    }{
      Ax \leq b,\,
      x \geq 0,\,
      x \in \Integers^n
    }
    \tag{ILP}
    \\
    \label{eq:chain-LP}
    & \qquad \leq
    \sup\setst{
      \iprodt{c}{x}
    }{
      Ax \leq b,\,
      x \geq 0,\,
      x \in \Reals^n
    }
    \tag{LP}
    \\
    \label{eq:chain-LD}
    & \qquad\qquad \leq
    \inf\setst{
      \iprodt{b}{y}
    }{
      A^{\transp} y \geq c,\,
      y \geq 0,\,
      y \in \Reals^m
    }
    \tag{LD}
    \\
    \label{eq:chain-ILD}
    & \qquad\qquad\qquad \leq
    \inf\setst{
      \iprodt{b}{y}
    }{
      A^{\transp} y \geq c,\,
      y \geq 0,\,
      y \in \Integers^m
    }.
    \tag{ILD}
  \end{align}
  If~\eqref{eq:chain-ILP} and~\eqref{eq:chain-ILD} are both feasible,
  the suprema and infima are attained, and the middle (second)
  inequality holds with equality.
\end{theorem}
\noindent (Attainment for~\eqref{eq:chain-ILP}
and~\eqref{eq:chain-ILD} follows from Meyer's
Theorem~\cite{Meyer74a}.)

Usually the feasible region of~\eqref{eq:chain-ILP} is contained
in~\(\set{0,1}^n\) and some optimal solution of~\eqref{eq:chain-ILD}
lies in~\(\set{0,1}^m\).  For instance, if \(G = (V,E)\) is a graph,
\(A\) is its \(V \times E\) incidence matrix, and both \(b\) and~\(c\)
are equal to the vector~\(\ones\) of all-ones, then
\eqref{eq:chain-ILP} formulates the maximum cardinality matching
problem and~\eqref{eq:chain-ILD} formulates the minimum cardinality
vertex cover problem.  Alternatively, if \(A\) is the \(E \times V\)
incidence matrix of~\(G\), we obtain the maximum cardinality stable
set problem and the minimum cardinality edge cover problem.  If \(A\)
is the clique-vertex incidence matrix of~\(G\), then
\eqref{eq:chain-ILP} still formulates the maximum cardinality stable
set problem, but now~\eqref{eq:chain-ILD} formulates the minimum
cardinality coloring problem.

What makes the conceptual framework brought forth by
\Cref{thm:LP-chain} so fundamental is the fact that, in many
interesting and important cases~\cite{Schrijver03a}, equality holds
throughout in the chain from~\Cref{thm:LP-chain}, which allows us to
refine a \emph{geometric} min-max relation (equality
between~\eqref{eq:chain-LP} and~\eqref{eq:chain-LD} given by LP
Strong~Duality) into a \emph{combinatorial} min-max relation (equality
between~\eqref{eq:chain-ILP} and~\eqref{eq:chain-ILD}).  For instance,
equality throughout holds for the first two cases described above
when~\(G\) is bipartite (and has no isolated vertices in the
second~case), thus proving very strong, weighted forms of Kőnig's
matching theorem and the Kőnig-Rado edge cover theorem.  In many
cases, the combinatorial optimality conditions thus obtained are
well-known to be key ingredients in the design of efficient algorithms
for solving the corresponding problems, both exactly and
approximately~\cite{Vazirani01a}.

Total dual integrality is arguably the most powerful and unifying
sufficient condition for equality throughout the chain
from~\Cref{thm:LP-chain}.  A vector in \(\Reals^n\) is \emph{integral}
if each of its components is an integer, and a rational system of
linear inequalities \(Ax\leq b\) is called \emph{totally dual integral
  (TDI)} if, for each integral vector~\(c \in \Integers^n\), the
linear program dual to \(\sup\setst{\iprodt{c}{x}}{Ax\leq b}\) has an
integral optimal solution whenever it has an optimal solution at~all.
In this case, if \(b\) itself is integral, then the polyhedron~\(P\)
determined by \(Ax \leq b\) is \emph{integral}, i.e., each nonempty
face of~\(P\) has an integral vector; thus, equality holds throughout
in the chain from~\Cref{thm:LP-chain}.  This was proved in seminal
work of Edmonds and Giles~\cite{EdmondsG77a} as a consequence of the
following fundamental result:
\begin{theorem}[Edmonds-Giles~\cite{EdmondsG77a}]
  \label{thm:Edmonds-Giles}
  If \(A \in \Rationals^{m \times n}\) and \(b \in \Rationals^m\)
  satisfy
  \(\sup\setst{
    \iprodt{c}{x}
  }{
    Ax \leq b
  }
  \in \Integers \cup \set{\pm\infty}
  \) for~each \(c \in \Integers^n\), then the polyhedron \(\setst{x
    \in \Reals^n}{Ax \leq b}\) is integral.
\end{theorem}

\begin{corollary}[Hoffman~\cite{Hoffman74a}]
  \label{cor:Hoffman}
  Let \(A \in \Rationals^{m \times n}\) and \(b \in \Rationals^m\).
  If \(P \coloneqq \setst{x \in \Reals^n}{Ax \leq b}\) is bounded and
  \(\max_{x \in P}\iprodt{c}{x} \in \Integers\) for each \(c \in
  \Integers^n\), then \(P\) is integral.
\end{corollary}

In the past couple of decades, it has become popular to formulate
combinatorial optimization problems using more general models of
convex optimization, with semidefinite programs (SDPs) playing a key
role.  Before we can proceed with our discussion, we need to introduce some
basic notation for SDPs.  The real vector space of symmetric \(n
\times n\) matrices is denoted by \(\Sym{n}\).  A matrix \(X \in
\Sym{n}\) is \emph{positive semidefinite} if \(\qform{X}{h} \geq 0\)
for every \(h \in \Reals^n\) or, equivalently, if every eigenvalue
of~\(X\) is nonnegative.  The \emph{semidefinite cone} is
\(\Psd{n} \coloneqq
\setst{
  X \in \Sym{n}
}{
  \text{\(X\) is positive semidefinite}
}
\).  The \emph{inner product} of \(X,Y \in \Sym{n}\) is \(\iprod{X}{Y}
\coloneqq \sum_{i=1}^n \sum_{j=1}^n X_{ij} Y_{ij}\).  Denote \([n]
\coloneqq \set{1,\dotsc,n}\) for each \(n \in \Naturals\).  We refer
the reader to
\Cref{tbl:special-sets,tbl:subsets,tbl:graphs,tbl:mat-vecs,tbl:opt}
and~\Cref{sec:notation} for the rest of the notation used throughout
the text.

When a combinatorial problem is formulated as in~\eqref{eq:chain-ILP},
the combinatorial objects are usually embedded in the (geometric)
space \(\Reals^n\) as incidence vectors, i.e., we consider the
feasible solutions to be of the form \(x = \incidvector{U}\), for
certain subsets \(U \subseteq [n]\), where for each \(i \in [n]\) the
\(i\)th coordinate of~\(\incidvector{U}\) is \(1\) if \(i \in U\)
and~\(0\) otherwise.  Having a correct ILP formulation for a
combinatorial optimization problem typically means that the feasible
solutions for~\eqref{eq:chain-ILP} are in \emph{exact} correspondence
with the combinatorial objects of interest in the problem.  One then
considers the LP relaxation~\eqref{eq:chain-LP} by dropping the
nonconvex constraint ``\(\,x \in \Integers^n\,\)''.  Note that the
``integer dual'' \eqref{eq:chain-ILD} is obtained from the
dual~\eqref{eq:chain-LD} of~\eqref{eq:chain-LP} by adding back the
nonconvex constraint ``\(\,y \in \Integers^m\,\)'' of the same form.

When embedding combinatorial objects into matrix space \(\Sym{n}\) for
an SDP formulation, one may embed a subset \(U \subseteq [n]\) as the
rank-one matrix \(X = \oprodsym{\incidvector{U}} \in \Psd{n}\).  It is
also common to use rank-one matrices arising from signed incidence
vectors, e.g., \(X = \oprodsym{s_U}\) where \(s_U =
2\incidvector{U}-\ones \in \set{\pm1}^n\) for some \(U \subseteq
[n]\).  (We~shall argue later that there is a ``better'' embedding,
which we shall adopt.)  One then obtains the following optimization
problems, partially mimicking the chain from~\Cref{thm:LP-chain}:
\begin{subequations}
  \label{eq:chain0}
  \begin{align}
    \label{eq:chain0-ISDP}
    &
    \sup\setst[\big]{
      \iprod{C}{X}
    }{
      \iprod{A_i}{X} \leq b_i\, \forall i \in [m],\,
      X \in \Psd{n},\,
      \rank(X) = 1
    }
    \\
    \label{eq:chain0-SDP}
    & \qquad\qquad\qquad\qquad
    \leq
    \sup\setst[\big]{
      \iprod{C}{X}
    }{
      \iprod{A_i}{X} \leq b_i\, \forall i \in [m],\,
      X \in \Psd{n}
    }
    \\
    \label{eq:chain0-SDD}
    & \qquad\qquad\qquad\qquad\qquad\qquad
    \leq
    \inf\setst[\big]{
      \iprodt{b}{y}
    }{
      y \in \Reals_+^m,\,
      {\textstyle\sum_{i=1}^m} y_i A_i - C \in \Psd{n}
    },
  \end{align}
\end{subequations}
where \(A_1,\dotsc,A_m,C \in \Sym{n}\) and \(b \in \Reals^m\).  Here
usually the feasible solutions for~\eqref{eq:chain0-ISDP} correspond
\emph{exactly} to the combinatorial objects of interest, as is the
case for~\eqref{eq:chain-ILP}.  Similarly as in~\Cref{thm:LP-chain},
the SDP relaxation~\eqref{eq:chain0-SDP} is obtained
from~\eqref{eq:chain0-ISDP} by dropping the nonconvex constraint
``\(\,\rank(X) = 1\,\)'', \eqref{eq:chain0-SDD} is the SDP dual
of~\eqref{eq:chain0-SDP}, and the last inequality is SDP Weak Duality.
There are many instances of the chain~\eqref{eq:chain0} in the
literature; see, e.g., \cite{GroetschelLS93a, GoemansW95a,
  Nesterov98a}.  Some of this work is in copositive programming (see,
for~instance, \cite{Burer09a} and the references therein).

Conspicuously missing from~\eqref{eq:chain0} is a fourth optimization
problem, that~is, an ``integer dual SDP'' corresponding
to~\eqref{eq:chain-ILD}.  In~fact, it is not even clear what the right
notion of integrality is for~\eqref{eq:chain0-SDD}, i.e., which
nonconvex constraint to add to~\eqref{eq:chain0-SDD} to obtain a
sensible combinatorial problem.  One could argue that we may just add
back the nonconvex constraint from~\eqref{eq:chain0-ISDP}, by
requiring the dual slack \({\textstyle\sum_{i=1}^m} y_i A_i - C\) to
have rank one, and it might also make sense to require the
vector~\(y\) to be integral.  Unfortunately, as we describe in
\Cref{sec:sdp-dual-int}, the ``integer dual SDP'' thus obtained is not
very satisfactory: whereas it can be made to generalize the
corresponding notion for LPs, it fails to provide sensible ``integer
duals'' for the SDP formulations of some of the most classical
combinatorial problems, namely the Lovász theta function for the
stable set problem and the Max~Cut~SDP\null.  Thus, we require our
notion of ``integrality constraints in the dual'' to provide
meaningful combinatorial min-max theorems at least for Max~Cut~SDP and
more importantly, for SDP formulations of the Lovász theta function.

In the late seventies, Lovász~\cite{Lovasz79a} solved a problem in
information theory by introducing the theta~function; this was one of
the earliest applications of semidefinite programming to combinatorial
optimization.  The~theta~function of a graph, which can be computed
efficiently (to within any desired precision), lies sandwiched between
its stability and clique-covering numbers; these latter parameters are
NP-hard to approximate~\cite{LundY94a, AroraLMSS98a, Hastad99a}, let
alone compute.  More importantly, there is a rich and elegant duality
theory centred around the theta function (see, e.g.,
\cite{CarliT17a}) and it has been used in many different
areas~\cite{FeigeL92a,Goemans97a,Simonyi01a,Lovasz03a}.  This rich and
elegant duality theory justifies why we take the underlying SDPs as
the main test case for any generalization of TDI theory.  The other
SDP mentioned above, the Max~Cut~SDP, was famously exploited in a
breakthrough approximation algorithm and its analysis by Goemans and
Williamson~\cite{GoemansW95a} and helped popularize SDP formulations
in the discrete optimization and theoretical computer science
communities.  This SDP remains fundamental due to its connections with
pioneering work in complexity theory related to the unique games
conjecture (see~\cite{Trevisan12a,KhotV15a}) and sums of
squares~\cite{BarakS14a}.

In this paper, we introduce a notion of integrality for SDPs that
\begin{enumerate}[(i)]
\item generalizes the usual rank-one constraint in primal SDPs;
\item allows us to extend the chain~\eqref{eq:chain0} so as to
  generalize \Cref{thm:LP-chain} for LPs in the natural, diagonal
  embedding of \(Ax \leq b\) into matrix space~\(\Sym{n}\);
\item is primal-dual symmetric;
\item yields sensible ``integer duals'' for the SDPs for the Lovász
  theta function and the Max Cut SDP.
\end{enumerate}
We use this integrality condition for SDPs to define the notion of
\emph{total dual integrality} for the defining system of an~SDP.  We
connect this new notion to~\Cref{cor:Hoffman} by extending the latter
to compact convex sets, using basic tools from convex analysis and ILP
theory, such as the Gomory-Chvátal closure.  We~prove that the total
dual integrality of an SDP formulation for the Lovász theta function
is equivalent to the underlying graph being perfect.  We also study a
close relative of TDIness for the Max Cut SDP and relate it to
bipartiteness of the underlying graph.  Along the way, we discuss an
intermediate generalization of TDIness for LPs in terms of lifted
(extended) formulations.  Finally, we discuss future research
directions along these lines, inspired by integrality (and other
exactness) notions in convex optimization.

In order to achieve this, several obstacles must be overcome.  First,
we must choose a specific format for SDPs that makes it natural to
work with integral solutions; that~is, we must settle for a specific
embedding of combinatorial objects into matrix space.  Note that this
is not an issue in the LP case, where incidence vectors are the most
natural choice of embedding.  We solve this partially by restricting
ourselves to binary integer programs, i.e., we only deal with integer
variables taking values in~\(\set{0,1}\); this is the usual case in
combinatorial optimization.  Our choice of embedding and our focus on
the combinatorial aspects of the dual SDP require us to rewrite SDP
constraints in a slightly unusual way; this happens because other
works in the literature do not focus on integrality for the dual SDP.
Finally, SDP formulations for combinatorial problems are usually
lifted formulations, so we must generalize the (algebraic) notion of
TDIness to these (geometric) extended formulations.

Some previous works on abstract notions of duality in the context of
integer programming are related to this one; we
highlight~\cite{CarvalhoT89a,RyanT94a}.

\bgroup
\renewcommand{\arraystretch}{1.2}
\begin{table}[!ht]
  \caption{Notation for special sets.}
  \centering
  \begin{tabular}{r c p{12cm} }
    \toprule
    \(\Integers_+\)
    & \(\coloneqq\) &
    \(\setst{x \in \Integers}{x \geq 0}\), the set of nonnegative integers
    \\
    \(\Reals_+\)
    & \(\coloneqq\) &
    \(\setst{x \in \Reals}{x \geq 0}\), the set of nonnegative reals
    \\
    \(\Reals_{++}\)
    & \(\coloneqq\) &
    \(\setst{x \in \Reals}{x > 0}\), the set of positive reals
    \\
    \([n]\)
    & \(\coloneqq\) &
    \(\set{1,\dotsc,n}\) for each \(n \in \Naturals\)
    \\
    \(\Sym{V}\)
    & \(\coloneqq\) &
    \(\setst{X \in \Reals^{V \times V}}{X = X^{\transp}}\),
    the real vector space of symmetric \(V \times V\) matrices
    \\
    \(\Psd{V}\)
    & \(\coloneqq\) &
    \(\setst{X \in \Sym{V}}{\qform{X}{h} \geq 0\,\forall h \in \Reals^V}\),
    the cone of positive semidefinite matrices in~\(\Sym{V}\)
    \\[2pt]
    \(\Symhat{\Vspc}\)
    & \(\coloneqq\) &
    \(\Sym{\zlift{V}}\), the lifted matrix space; see \eqref{eq:4}
    \\
    \(\Psdhat{\Vspc}\)
    & \(\coloneqq\) &
    \(\Psd{\zlift{V}}\), the semidefinite cone in the lifted space; see \eqref{eq:4}
    \\[2pt]
    \(\Symnonneg{V}\)
    & \(\coloneqq\) &
    \(\setst{X \in \Sym{V}}{X \geq 0}\),
    the cone of entrywise nonnegative matrices in~\(\Sym{V}\)
    \\
    \bottomrule
  \end{tabular}
  \label{tbl:special-sets}
\end{table}
\egroup                         

\bgroup
\renewcommand{\arraystretch}{1.2}
\begin{table}[!ht]
  \caption{Notation for sets.}
  \centering
  \begin{tabular}{r c p{12cm} }
    \toprule
    \(\Powerset{V}\)
    & \(\coloneqq\) &
    the power set of~\(V\)
    \\[4pt]
    \(\tbinom{V}{k}\)
    & \(\coloneqq\) &
    \(\setst{U \subseteq V}{\card{U} = k}\), the collection of \emph{\(k\)-subsets} of~\(V\)
    \\[4pt]
    \(\tbinom{V}{i \in}\)
    & \(\coloneqq\) &
    the collection of subsets of~\(V\) that contain \(i \in V\)
    \\[4pt]
    \(\tbinom{V}{ij \subseteq}\)
    & \(\coloneqq\) &
    the collection of subsets of~\(V\) that contain \(i \in V\) and \(j \in V\)
    \\[4pt]
    \(f\mathord{\restriction}_U\)
    & \(\coloneqq\) &
    the restriction of the function \(f \colon V \to W\) to \(U \subseteq V\)
    \\
    \(ij\)
    & \(\coloneqq\) &
    \(\set{i,j}\) or \((i,j)\),
    whichever parses
    \\
    \bottomrule
  \end{tabular}
  \label{tbl:subsets}
\end{table}
\egroup                         

\bgroup
\renewcommand{\arraystretch}{1.2}
\begin{table}[!ht]
  \caption{Notation for a graph \(G = (V,E)\).}
  \centering
  \begin{tabular}{r c p{12cm} }
    \toprule
    \(\overline{G}\)
    & \(\coloneqq\) &
    \((V,\overline{E})\) where \(\overline{E} \coloneqq \tbinom{V}{2}
    \setminus E\), i.e., the \emph{complement} of~\(G\)
    \\
    \(G[U]\)
    & \(\coloneqq\) &
    \(\paren[\big]{U,E\cap\tbinom{U}{2}}\), i.e.,
    the subgraph of \(G\) \emph{induced} by \(U \subseteq V\)
    \\
    \(K_U\)
    & \(\coloneqq\) &
    the complete graph on vertex set \(U\)
    \\
    \(\Ascr_G\)
    & \(\coloneqq\) &
    \(\setst{A \in \Sym{V}}{A_{ij} \neq 0 \implies ij \in E}\),
    the set of \emph{weighted adjacency matrices} of~\(G\)
    \\
    \(\Kcal(G)\)
    & \(\coloneqq\) &
    the set of cliques of~\(G\)
    \\
    \(\omega(G)\)
    & \(\coloneqq\) &
    \(\max\setst{\card{K}}{K \in \Kcal(G)}\), i.e.,
    the \emph{clique number} of~\(G\)
    \\
    \(\chi(G)\)
    & \(\coloneqq\) &
    \(\min\setst{\card{\Pcal}}{\text{\(\Pcal\) a partition of~\(V\) into
        stable sets}}\), i.e.,
    the \emph{chromatic number} of~\(G\)
    \\
    \(\alpha(G,w)\)
    & \(\coloneqq\) &
    the (weighted) stability number of~\(G\) with weights \(w \colon V \to
    \Reals\); see~\eqref{eq:def-alpha}
    \\
    \(\theta(G,w)\)
    & \(\coloneqq\) &
    the Lovász theta number of~\(G\) with weights \(w \colon V \to
    \Reals\); see~\eqref{eq:theta-SDP}
    \\
    \(\theta'(G,w)\)
    & \(\coloneqq\) &
    the variant of~\(\theta(G,w)\) defined in~\eqref{eq:theta'-SDP}
    \\
    \(\theta^+(G,w)\)
    & \(\coloneqq\) &
    the variant of~\(\theta(G,w)\) defined in~\eqref{eq:theta+-SDP}
    \\
    \(\overline{\chi}(G,w)\)
    & \(\coloneqq\) &
    the (weighted) clique covering number of~\(G\) with weights \(w \colon V \to
    \Reals\); see~\eqref{eq:clique-covering}
    \\
    \(\Laplacian{G}\)
    & \(\coloneqq\) &
    the weighted Laplacian of~\(G\); see~\eqref{eq:def-Laplacian}
    \\
    \(\delta(U)\)
    & \(\coloneqq\) &
    the cut in~\(G\) with shore \(U \subseteq V\); see~\eqref{def:cut}
    \\
    \(\delta(i)\)
    & \(\coloneqq\) &
    \(\delta(\set{i})\) for a vertex \(i \in V\)
    \\
    \(N(i)\)
    & \(\coloneqq\) &
    \(\setst{j \in V}{ij \in \delta(i)}\), i.e., the set of neighbors
    of~\(i\) in~\(G\)
    \\
    \bottomrule
  \end{tabular}
  \label{tbl:graphs}
\end{table}
\egroup                         

\bgroup
\renewcommand{\arraystretch}{1.2}
\begin{table}[!ht]
  \caption{Notation for vectors and matrices.}
  \centering
  \begin{tabular}{r c p{12cm} }
    \toprule
    \(\setst{e_i}{i \in V}\)
    & \(\coloneqq\) &
    the canonical basis of \(\Reals^V\)
    \\
    \(\trace(A)\)
    & \(\coloneqq\) &
    \(\sum_{i \in V} A_{ii}\),
    the \emph{trace} of \(A \in \Reals^{V \times V}\)
    \\
    \(\iprod{X}{Y}\)
    & \(\coloneqq\) &
    \(\trace(XY^{\transp})\),
    the \emph{(trace) inner-product} on~\(\Reals^{V \times V}\)
    \\
    \(\Acal^*\)
    & \(\coloneqq\) &
    the adjoint of a linear map \(\Acal\) between real inner-product spaces
    \\
    \(\diag\)
    & \(\coloneqq\) &
    the linear map from \(\Reals^{V \times V}\) to \(\Reals^V\)
    that extracts the diagonal of a matrix
    \\
    \(\Diag\)
    & \(\coloneqq\) &
    the adjoint of \(\diag\) from \(\Reals^{V}\) to \(\Reals^{V \times V}\),
    which builds diagonal matrices
    \\
    \(X[U]\)
    & \(\coloneqq\) &
    \(X\mathord{\restriction}_{U \times U} \in \Reals^{U \times U}\), i.e.,
    the \emph{principal submatrix} of \(X \in \Reals^{V \times V}\) indexed
    by \(U \subseteq V\)
    \\
    \(\ones\)
    & \(\coloneqq\) &
    the vector of all-ones in the appropriate space
    \\
    \(\incidvector{U}\)
    & \(\coloneqq\) &
    the \emph{incidence vector} of \(U \subseteq V\) in
    \(\set{0,1}^V\); see~\eqref{eq:2}
    \\
    \(I\)
    & \(\coloneqq\) &
    the identity matrix in appropriate dimension
    \\
    \(\geq\)
    & \(\coloneqq\) &
    the nonnegative partial order on \(\Reals^{V \times W}\), i.e.,
    \(A \geq B\) if \(A_{ij} \geq B_{ij}\) \(\forall (i,j) \in V \times W\)
    \\[4pt]
    \(\succeq\)
    & \(\coloneqq\) &
    the Löwner partial order on \(\Sym{V}\), i.e.,
    \(A \succeq B \iff A - B \in \Psd{V}\)
    \\
    \(\sqrt{w}\) & \(\coloneqq\) & the componentwise square root of
    \(w \in \Reals_+^V\), i.e., \((\sqrt{w}\,)_i \coloneqq
    \sqrt{w_i}\) for every \(i \in V\)
    \\
    \(x \oplus y\)
    & \(\coloneqq\) &
    the direct sum of vectors \(x \in \Reals^V\) and \(y \in \Reals^W\)
    \\
    \(x \hprod y\) & \(\coloneqq\) & the \emph{Hadamard product} of
    \(x, y \in \Reals^V\), i.e., \((x \hprod y)_i \coloneqq
    x_i y_i\) for every \(i \in V\)
    \\
    \(\supp(x)\)
    & \(\coloneqq\) &
    \(\setst{i \in V}{x_i \neq 0}\),
    the \emph{support} of \(x \in \Reals^V\)
    \\
    \(\Symmetrize(A)\)
    & \(\coloneqq\) &
    \(\tfrac{1}{2}(A+A^{\transp})\),
    the orthogonal projection of \(A \in \Reals^{V \times V}\) into \(\Sym{V}\)
    \\
    \bottomrule
  \end{tabular}
  \label{tbl:mat-vecs}
\end{table}
\egroup                         

\bgroup
\renewcommand{\arraystretch}{1.2}
\begin{table}[!ht]
  \caption{Notation for (convex) optimization, with \(\Cscr\) a convex
    subset of an Euclidean space~\(\mathbb{E}\).}
  \centering
  \begin{tabular}{r c p{12cm} }
    \toprule
    \(\conv(X)\)
    & \(\coloneqq\) &
    the convex hull of \(X \subseteq \mathbb{E}\)
    \\
    \(\suppf{\Cscr}{w}\)
    & \(\coloneqq\) &
    the support function of~\(\Cscr\) at \(w \in \mathbb{E}\);
    see~\eqref{def:suppf}
    \\
    \(\GomoryChvatal(\Cscr)\)
    & \(\coloneqq\) &
    the Gomory-Chvátal closure of~\(\Cscr\);
    see~\eqref{def:Gomory-Chvatal}
    \\
    \(\Cscr_I\)
    & \(\coloneqq\) &
    \(\conv(\Cscr \cap \Integers^n)\),
    i.e., the \emph{integer hull} of~\(\Cscr\)
    \\
    \(\Cscr^{\circ}\)
    & \(\coloneqq\) &
    \(\setst{y \in \mathbb{E}}{\iprod{y}{x} \leq 1\,\forall x \in \Cscr}\)
    i.e., the \emph{polar} of~\(\Cscr\)
    \\
    \bottomrule
  \end{tabular}
  \label{tbl:opt}
\end{table}
\egroup                         

\subsection{Notation}
\label{sec:notation}

We use Iverson's notation: for a predicate \(P\), we denote
\begin{equation}
  \label{eq:1}
  [P] \coloneqq
  \begin{dcases*}
    1 & if \(P\) holds;\\
    0 & otherwise.
  \end{dcases*}
\end{equation}
When \(P\) is false, \([P]\) is considered ``strongly zero'', in the
sense that \([P]\) is allowed to multiply a meaningless term and the
result will be zero.  The simplest example of this is that \([\alpha >
0]\tfrac{1}{\alpha}\) is taken to be \(0\) if \(\alpha = 0\).

Throughout the text, \(V\) should be considered a finite set, usually
taken to be the vertex set of a graph \(G = (V,E)\); \emph{all graphs
  in this paper are simple}.  The \emph{incidence vector} of \(U
\subseteq V\) is \(\incidvector{U} \in \set{0,1}^V\) defined as
\begin{equation}
  \label{eq:2}
  (\incidvector{U})_i \coloneqq [i \in U]
  \qquad
  \forall i \in V.
\end{equation}
The rest of our notation is mostly standard, and it can be looked up
in~\Cref{tbl:special-sets,tbl:subsets,tbl:graphs,tbl:mat-vecs,tbl:opt}.

\subsection{Organization}
\label{sec:org}

The rest of this text is organized as follows.  We discuss dual
integrality constraints for SDPs in~\Cref{sec:sdp-dual-int}, including
drawbacks of the rank-one constraint usually added to the primal SDP
(further drawbacks are postponed to~\Cref{sec:rank-maxcut}
and~\Cref{sec:rank-theta-tr}), as~well as embedding issues.  There, we
show that our notion of dual integrality befits nicely with the Lovász
theta function.  In~\Cref{sec:int-convex}, we~generalize
\Cref{cor:Hoffman}, which motivates us to define a notion of
\emph{total} dual integrality for SDPs in~\Cref{sec:sdp-tdi}; we show
that the latter is sufficient for \emph{primal} integrality.
In~\Cref{sec:theta-tdi}, we characterize total dual integrality for
formulations of the Lovász theta function and we study dual
integrality for the MaxCut SDP with nonnegative weight functions
in~\Cref{sec:maxcut}.  (Some of the limitations in our theory as
applied to the MaxCut SDP, related the use of nonnegative weight
functions, are discussed in~\Cref{sec:maxcut-nonneg}.)  We conclude
our paper with several open problems and future research directions
in~\Cref{sec:conclusion}.

\section{Fundamental Framework and Integrality Constraint for Dual
  SDP}
\label{sec:sdp-dual-int}

We discuss in \Cref{sec:drawbacks-rank-one} below the shortcomings of
the rank constraint as an ``integrality constraint'' for the dual
SDP~\eqref{eq:chain0-SDD}, and propose a replacement
in~\Cref{sec:integrality-constraint}.  Along the discussion, a few,
somewhat unusual choices will be made, which are not normally done in
the SDP literature; e.g., we are careful when writing linear
inequalities of the form \(\iprod{A}{X} \leq \beta\) on a matrix
variable~\(X\) with an \emph{integral symmetric} matrix~\(A\) and
\emph{integer}~\(\beta\).  The reason we insist on symmetry of~\(A\)
is to properly set up the dual SDP, and we want~\(A\) and~\(\beta\) to
be integral so as to simplify combinatorial interpretation of the
linear system; this is also the case when one studies the ILP chain
from~\Cref{thm:LP-chain} in the context of classical TDIness theory.

\subsection{Drawbacks of the Rank-one Constraint as a Dual Integrality
  Constraint}
\label{sec:drawbacks-rank-one}

In order to discuss integrality constraints for SDPs, we must first
choose a standard form to embed combinatorial objects (e.g., subsets
of some finite ground set~\(V\)) into matrix space \(\Sym{V}\).  The
format we shall choose actually embeds subsets of a finite set~\(V\)
as matrices in \(\Sym{\zlift{V}}\), i.e., the index set has one extra
element, which we call~\(0\), \emph{assumed throughout not to be
  in~\(V\)}.  Each subset \(U\) of~\(V\) is embedded as the rank-one
matrix
\begin{equation}
  \label{eq:3}
  \Xhat
  \coloneqq
  \oprodsym{
    \begin{bmatrix}
      1               \\
      \incidvector{U} \\
    \end{bmatrix}
  }
  =
  \begin{bmatrix}
    1               & \incidvector{U}^{\transp}  \\[3pt]
    \incidvector{U} & \oprodsym{\incidvector{U}} \\
  \end{bmatrix}
  \in
  \Psd{\zlift{V}};
\end{equation}
as a convention, \emph{we decorate matrices in this lifted space with
  a hat}, e.g., \(\Xhat\) in~\eqref{eq:3}.  Similarly, since we use
the lifted matrix space so often, we shall abbreviate
\begin{equation}
  \label{eq:4}
  \Symhat{\Vspc} \coloneqq \Sym{\zlift{V}}
  \qquad\text{and}\qquad
  \Psdhat{\Vspc} \coloneqq \Psd{\zlift{V}},
\end{equation}
and we also decorate subsets of~\(\Symhat{\Vspc}\) with a hat, e.g.,
\(\Cscrhat \subseteq \Symhat{\Vspc}\).  By writing any matrix~\(\Xhat\)
from~\eqref{eq:3} in the form
\begin{equation}
  \label{eq:5}
  \Xhat
  =
  \begin{bmatrix}
    1 & x^{\transp}\, \\
    x & X             \\
  \end{bmatrix}
  \in \Symhat{\Vspc},
\end{equation}
with \(X \in \Sym{V}\), one sees that it satisfies the linear constraints
\begin{equation}
  \label{eq:6}
  \Xhat_{00} = 1
  \qquad
  \text{and}
  \qquad
  x_j = X_{jj} \geq 0
  \quad
  \forall j \in V,
\end{equation}
which we shall write as
\begin{subequations}
  \label{eq:SDP-01}
  \begin{alignat}{3}
    \iprod{\oprodsym{e_0}}{\Xhat} & = 1, & \qquad &
    \label{eq:SDP-01-00}
    \\
    \iprod[\big]{
      2\Symmetrize(\oprod{e_j}{(e_j-e_0)})
    }{
      \Xhat
    } & = 0 & \qquad & \forall j \in V,
    \label{eq:SDP-01-diag}
    \\
    \iprod[\big]{
      \oprodsym{e_j}
    }{
      \Xhat
    } & \geq 0 & \qquad & \forall j \in V.
    \label{eq:SDP-01-nonneg}
  \end{alignat}
\end{subequations}
The constraints~\eqref{eq:SDP-01}, together with the constraint
\(\rank(\Xhat) = 1\), ensure that \(\Xhat\) has the form~\eqref{eq:3}
for some \(U \subseteq V\).  Throughout the rest of the text, one may
think that every system of linear inequalities on~\(\Xhat\) arising
from combinatorial problems includes the
constraints~\eqref{eq:SDP-01}, just as one usually considers the
linear constraints \(Ax \leq b\), \(x \geq 0\)
from~\eqref{eq:chain-ILP} to include \(0 \leq x \leq \ones\).

Another constraint satisfied by \(\Xhat\) of the form~\eqref{eq:3},
using the notation of~\eqref{eq:5}, is \(X = \Xhat[V] \geq 0\).
Sometimes it will make sense to add this extra constraint
to~\eqref{eq:SDP-01}, leading to the following constraints:
\begin{subequations}
  \label{eq:SDP-01geq}
  \begin{alignat}{3}
    \iprod{\oprodsym{e_0}}{\Xhat} & = 1, & \qquad &
    \label{eq:SDP-01geq-00}
    \\
    \iprod[\big]{
      2\Symmetrize(\oprod{e_j}{(e_j-e_0)})
    }{
      \Xhat
    } & = 0 & \qquad & \forall j \in V,
    \label{eq:SDP-01geq-diag}
    \\
    \iprod{\oprodsym{e_j}}{\Xhat}
    & \geq 0 & \qquad & \forall j \in V,
    \label{eq:SDP-01geq-geq1}
    \\
    \iprod[\big]{2\Symmetrize(\oprod{e_i}{e_j})}{\Xhat}
    & \geq 0 & \qquad & \forall i,j \in V,\text{ such that }i \neq j.
    \label{eq:SDP-01geq-geq2}
  \end{alignat}
\end{subequations}

The embedding described above is used in some formulations of the
theta function (see~\cite{GroetschelLS93a,Schrijver03a}), in~the
lift-and-project hierarchies of Lovász and Schrijver~\cite{LovaszS91a}
and Lasserre~\cite{Lasserre01a} (also see Laurent~\cite{Laurent03a}), and in
copositive formulations for mixed integer linear programs by
Burer~\cite{Burer09a}.

A simple, natural way to obtain an SDP relaxation for~\eqref{eq:chain-ILP} is to formulate
\begin{subequations}
  \label{eq:LP-as-SDP}
  \begin{alignat}{3}
    \text{Maximize} \quad & \iprod{\Diag(0 \oplus c)}{\Xhat} & \\[0pt]
    \label{eq:LP-as-SDP-embedding}
    \text{subject to}\quad & \text{\(\Xhat\)
      satisfies~\eqref{eq:SDP-01geq} with \(V \coloneqq [n]\)},
    & \\[0pt]
    & \iprod{
      \Diag(-b_i \oplus A^{\transp}e_i)
    }{
      \Xhat
    } \leq 0
    & \qquad & \forall i \in [m], \\[0pt]
    & \Xhat \in \Psdhat{n}. &
  \end{alignat}
\end{subequations}
In this case, to obtain an exact reformulation
of~\eqref{eq:chain-ILP}, corresponding to~\eqref{eq:chain0-ISDP}, one
may add the rank constraint \(\rank(\Xhat) \leq 1\)
to~\eqref{eq:LP-as-SDP}.  Note, however, that~\eqref{eq:LP-as-SDP} is
a potentially tighter relaxation for~\eqref{eq:chain-ILP}
than~\eqref{eq:chain-LP}.  The~SDP~dual to~\eqref{eq:LP-as-SDP} may be
written as
\begin{subequations}
  \label{eq:LP-as-SDP-dual}
  \begin{alignat}{3}
    \text{Minimize} \quad & \eta & \\[0pt]
    \text{subject to}\quad &
    \begin{bmatrix}
      \eta & -u^{\transp}  \\
      -u   & \Diag(2u) - Z \\
    \end{bmatrix}
    + \sum_{i \in [m]} y_i
    \begin{bmatrix}
      -b_i & 0^{\transp}           \\
      0    & \Diag(A^{\transp}e_i) \\
    \end{bmatrix}
    - \Shat =
    \begin{bmatrix}
      0 & 0^{\transp} \\
      0 & \Diag(c)    \\
    \end{bmatrix},
    \label{eq:7}
    \\
    &
    \Shat \in \Psdhat{n},\,
    \eta \in \Reals,\,
    u \in \Reals^n,\,
    y \in \Reals_+^m,\,
    Z \in \Symnonneg{n}. &
  \end{alignat}
\end{subequations}
If~\eqref{eq:LP-as-SDP-embedding} is weakened to ``\(\Xhat\)
satisfies~\eqref{eq:SDP-01}'', again with \(V = [n]\), then the
variable \(Z\) in~\eqref{eq:LP-as-SDP-dual} would be required to take
the form \(Z = \Diag(z)\) for some \(z \in \Reals_+^n\).

It is easy to check that, if \(y\) is feasible
in~\eqref{eq:chain-ILD}, then \((\eta,Z,y,\Shat,u) \coloneqq
(\iprodt{b}{y},\Diag(A^{\transp}y - c),y,0,0)\) is feasible
in~\eqref{eq:LP-as-SDP-dual} with the same objective value as that
of~\(y\) in \eqref{eq:chain-ILD}.  Thus, the rank constraint
\(\rank(\Shat) \leq 1\) seems reasonable as an integrality constraint
for~\eqref{eq:LP-as-SDP-dual}.  In~fact, we may even consider the
tighter rank constraint \(\rank(\Shat) = 1\), as long as we allow
\(\eta\) to take on real values (rather than only integral ones),
possibly at the cost of nonattainment.

Now we move on to the SDP formulation for~\(\theta\), the Lovász theta
function.  In~fact, we will also consider variations of~\(\theta\)
usually denoted by~\(\theta'\) and~\(\theta^+\), which were introduced
independently by McEliece, Rodemich, and Rumsey~\cite{McElieceRR78a}
and Schrijver~\cite{Schrijver79a}, and by Szegedy~\cite{Szegedy94a},
respectively.  \emph{We shall show that the rank constraint is very
  inadequate for the dual SDP in this setting, for all three
  variants.}

Let \(G = (V,E)\) be a graph and let \(w \colon V \to \Reals\).  There
are several equivalent formulations for the \emph{weighted theta
  number}~\(\theta(G;w)\) of~\(G\) with weights~\(w\) (see, e.g.,
\cite{CarliT17a}), and similarly for its variations~\(\theta'(G;w)\)
and~\(\theta^+(G;w)\).  In view of our choice of format for SDPs that
includes the constraints~\eqref{eq:SDP-01}, we shall use the following
formulation for \(\theta(G;w)\):
\begin{subequations}
  \label{eq:theta-SDP}
  \begin{alignat}{3}
    \text{Maximize} \quad & \iprod{\Diag(0 \oplus w)}{\Xhat} & \\[0pt]
    \label{eq:8}
    \text{subject to}\quad & \text{\(\Xhat\) satisfies~\eqref{eq:SDP-01}}, & \\[0pt]
    \label{eq:9}
    & \iprod{2\Symmetrize(\oprod{e_i}{e_j})}{\Xhat} = 0 & \qquad & \forall ij \in E, \\[0pt]
    & \Xhat \in \Psdhat{\Vspc}. &
  \end{alignat}
\end{subequations}
Note that, if \(U \subseteq V\) is \emph{stable} in~\(G\), i.e., no
edge of~\(G\) has both endpoints in~\(U\), then the matrix \(\Xhat\)
defined in~\eqref{eq:3} is feasible in~\eqref{eq:theta-SDP} with
objective value \(\iprodt{w}{\incidvector{U}} = \sum_{u \in U} w_u\).

We formulate \(\theta'(G,w)\) as
\begin{subequations}
  \label{eq:theta'-SDP}
  \begin{alignat}{3}
    \text{Maximize} \quad & \iprod{\Diag(0 \oplus w)}{\Xhat} & \\[0pt]
    \text{subject to}\quad & \text{\(\Xhat\) satisfies~\eqref{eq:SDP-01}}, & \\[0pt]
    & \iprod{2\Symmetrize(\oprod{e_i}{e_j})}{\Xhat} = 0 & \qquad & \forall ij \in E, \\[0pt]
    & \iprod{2\Symmetrize(\oprod{e_i}{e_j})}{\Xhat} \geq 0 & \qquad & \forall ij \in \overline{E}, \\[0pt]
    & \Xhat \in \Psdhat{\Vspc}, &
  \end{alignat}
\end{subequations}
where \(\overline{E} \coloneqq \tbinom{V}{2} \setminus E\), and
\(\theta^+(G,w)\) is formulated as
\begin{subequations}
  \label{eq:theta+-SDP}
  \begin{alignat}{3}
    \text{Maximize} \quad & \iprod{\Diag(0 \oplus w)}{\Xhat} & \\[0pt]
    \text{subject to}\quad & \text{\(\Xhat\) satisfies~\eqref{eq:SDP-01}}, & \\[0pt]
    & \iprod{2\Symmetrize(\oprod{e_i}{e_j})}{\Xhat} \leq 0 & \qquad & \forall ij \in E, \\[0pt]
    & \Xhat \in \Psdhat{\Vspc}. &
  \end{alignat}
\end{subequations}

The dual SDP of~\eqref{eq:theta'-SDP} is:
\begin{subequations}
  \label{eq:SDP-theta'-dual}
  \begin{alignat}{3}
    \text{Minimize} \quad & \eta & \\[0pt]
    \text{subject to}\quad &
    \begin{bmatrix}
      \eta & -u^{\transp}    \\
      -u   & \Diag(2u-z) \\
    \end{bmatrix}
    + \sum_{ij \in \tbinom{V}{2}} y_{ij}
    \begin{bmatrix}
      0 & 0^{\transp}                    \\
      0 & 2\Symmetrize(\oprod{e_i}{e_j}) \\
    \end{bmatrix}
    - \Shat =
    \begin{bmatrix}
      0 & 0^{\transp} \\
      0 & \Diag(w)    \\
    \end{bmatrix},
    \label{eq:SDP-theta'-dual-eq}
    \\
    &
    \Shat \in \Psdhat{\Vspc},\,
    \eta \in \Reals,\,
    u \in \Reals^V,\,
    z \in \Reals_+^V,\,
    y \in \Reals^E \oplus -\Reals_+^{\overline{E}}. &
  \end{alignat}
\end{subequations}
Note that the dual for the formulation~\eqref{eq:theta-SDP}
of~\(\theta(G;w)\) is similar, except that it requires
\(y\mathord{\restriction}_{\overline{E}} = 0\), and the dual for the
formulation~\eqref{eq:theta+-SDP} of~\(\theta^+(G;w)\) furthermore has
the sign constraint \(y\mathord{\restriction}_E \geq 0\).

We claim that,
\begin{equation}
  \label{eq:10}
  \text{
    if \eqref{eq:SDP-theta'-dual} has a feasible solution with
    \(\rank(\Shat) \leq 1\) and \(w \in \Reals_{++}^V\), then \(G =
    K_V\).
  }
\end{equation}
Indeed, suppose that \(\rank(\Shat) \leq 1\).  We have \(\eta > 0\)
by weak duality, so \(\rank(\Shat)=1\) and
\[
\Shat =
\begin{bmatrix}
  \eta & -u^{\transp}                \\
  -u   & \tfrac{1}{\eta}\oprodsym{u} \\
\end{bmatrix}.
\]
Then,
\begin{equation}
  \label{eq:11}
  \Diag(2u-z-w) + \sum_{ij \in \tbinom{V}{2}} 2 y_{ij}
  \Symmetrize(\oprod{e_i}{e_j}) = \tfrac{1}{\eta} \oprodsym{u}.
\end{equation}
By applying~\(\diag\) to both sides of~\eqref{eq:11}, we get \(2u-z-w
= u \hprod u\) so \(2u = (u \hprod u) + z + w \in \Reals_{++}^V\).
Next let \(i,j \in V\) be distinct.  The \(ij\)th entry
of~\eqref{eq:11} is \(y_{ij} = \tfrac{1}{\eta} u_i u_j > 0\) whence
\(ij \in E\).  This proves~\eqref{eq:10}.  Hence, the dual SDPs for
the formulations of all three variants of~\(\theta\) only have
feasible solutions with rank-one slacks if~\(G\) is complete.

One might argue that we have chosen an inappropriate formulation for
the rank constraint.  However, given the mandatory
constraints~\eqref{eq:SDP-01}, the formulation above is the most
natural one.  For completeness, we~show in~\Cref{sec:rank-theta-tr}
that the rank constraint is not adequate either for the more popular
formulation of the theta function with variable \(X \in \Psd{V}\) and
the trace constraint \(\trace(X) = 1\); in~\Cref{sec:rank-maxcut}, we
also treat the rank constraint for the dual of the MaxCut SDP, which
will be introduced in~\Cref{sec:maxcut}.

\subsection{An Improved Dual Integrality Constraint}
\label{sec:integrality-constraint}

In view of our adopted embedding~\eqref{eq:3}, let us draft the
complete version of the (partial) chain of
inequalities~\eqref{eq:chain0} as
\begin{subequations}
  \label{eq:chain1}
  \begin{align}
    &
    \sup\setst[\big]{
      \iprod{\Chat}{\Xhat}
    }{
      \iprod{\Ahat_i}{\Xhat} \leq b_i\, \forall i \in [m],\,
      \Xhat \in \Psdhat{n},\,
      \text{``\(\Xhat\) integral''}
    }
    \label{eq:chain1-ISDP}
    \\
    & \qquad\qquad\qquad\qquad
    \leq
    \sup\setst[\big]{
      \iprod{\Chat}{\Xhat}
    }{
      \iprod{\Ahat_i}{\Xhat} \leq b_i\, \forall i \in [m],\,
      \Xhat \in \Psdhat{n}
    }
    \label{eq:chain1-SDP}
    \\
    & \qquad\qquad\qquad\qquad\qquad\qquad
    \leq
    \inf\setst[\big]{
      \iprodt{b}{y}
    }{
      y \in \Reals_+^m,\,
      {\Shat = \textstyle\sum_{i=1}^m} y_i \Ahat_i - \Chat \in \Psdhat{n}
    }
    \label{eq:chain1-SDD}
    \\
    & \qquad\qquad\qquad\qquad\qquad\qquad\quad
    \leq
    \inf\setst[\big]{
      \iprodt{b}{y}
    }{
      y \in \Integers_+^m,\,
      {\Shat = \textstyle\sum_{i=1}^m} y_i \Ahat_i - \Chat \in \Psdhat{n},\,
      \text{``\(\Shat\) integral''}
    }.
    \label{eq:chain1-ISDD}
  \end{align}
\end{subequations}
Assume that the system \(\iprod{\Ahat_i}{\Xhat} \leq b_i,\, i \in
[m]\), includes the constraints~\eqref{eq:SDP-01}.

To define the \emph{integrality constraint}
for~\eqref{eq:chain1-ISDD}, we shall consider the dual slack \(\Shat =
\sum_{i=1}^m y_i \Ahat_i - \Chat\).
\begin{definition}
  \label{def:dual-int}
  Let \(\Shat\) be feasible in~\eqref{eq:chain1-SDD}.  We say that
  ``\(\Shat\)~is~\emph{integral}\,'' if \(\Shat\) is a sum \(\Shat =
  \sum_{k=1}^N \Shat_k\) of rank-one matrices \(\Shat_1,\dotsc,\Shat_N
  \in \Psdhat{n}\) such that, for each \(k \in [N]\), we have
  \begin{subequations}
    \label{eq:sdp-linear-int-dual}
    \begin{alignat}{3}
      \iprod{\oprodsym{e_0}}{\Shat_k} & = 1, & \qquad &
      \label{eq:sdp-linear-int-dual-0}
      \\
      \iprod[\big]{2\Symmetrize(\oprod{e_j}{(e_j+e_0)})}{\Shat_k} & = 0
      & \quad & \forall j \in V.
      \label{eq:sdp-linear-int-dual-1}
    \end{alignat}
  \end{subequations}
\end{definition}
Note that this is almost identical to the constraints
in~\eqref{eq:SDP-01}, except for the sign of~\(e_0\)
in~\eqref{eq:sdp-linear-int-dual-1}.  Equivalently, each \(\Shat_k\)
must have the form
\begin{equation*}
  \Shat_k =
  \begin{bmatrix}
    1    & -s_k^{\transp}\, \\
    -s_k & S_k              \\
  \end{bmatrix}
\end{equation*}
and satisfy \(\diag(S_k) = s_k\).  Since \(\Shat_k\) has rank-one, we
must have \(S_k = \oprodsym{s_k}\).  Hence, the condition
``\(\Shat\)~is~\emph{integral}\,'' may be interpreted with a more
combinatorial flavor as requiring \(\Shat\) to have the form
\begin{equation*}
  \Shat = \sum_{K \in \mathcal{K}}
  \oprodsym{
    \begin{bmatrix}
      -1              \\
      \incidvector{K} \\
    \end{bmatrix}
  }
\end{equation*}
for some \emph{family} (i.e., multiset) \(\mathcal{K}\) of subsets
of~\([n]\).  Denote the power set of a set~\(V\) by \(\Powerset{V}\).
By letting \(m \colon \Powerset{V} \to \Integers_+\) denote the
multiplicity of each subset \(K \subseteq V \coloneqq [n]\)
in~\(\mathcal{K}\), we may rewrite the condition
``\(\Shat\)~is~integral\,'' as
\begin{equation}
  \label{eq:dual-int}
  \Shat = \sum_{A \subseteq V} m_A
  \begin{bmatrix}
    1                & -\incidvector{A}^{\transp} \\[3pt]
    -\incidvector{A} & \oprodsym{\incidvector{A}} \\
  \end{bmatrix}
  \text{ for some }
  m \colon \Powerset{V} \to \Integers_+.
  \tag{D\(\Integers\)}
\end{equation}

The \emph{integrality constraint} for~\eqref{eq:chain1-ISDP} is
analogous:
\begin{definition}
  \label{def:primal-int}
  Let \(\Xhat\) be feasible in~\eqref{eq:chain1-SDP}.  We say that
  ``\(\Xhat\)~is~\emph{integral}\,'' if \(\Xhat\) is a sum \(\Xhat =
  \sum_{k=1}^N \Xhat_k\) of rank-one matrices \(\Xhat_1,\dotsc,\Xhat_N
  \in \Psdhat{n}\) such that \(\Xhat_k\) satisfies~\eqref{eq:SDP-01}
  for each \(k \in [N]\).
\end{definition}
The usual rank constraint ``\(\rank(\Xhat) = 1\)'', which is the usual
notion of integrality for~\(\Xhat\), can be simply enforced by the
linear constraint \(\Xhat_{00} = 1\).  As before, this integrality
constraint can be described as
\begin{equation}
  \label{eq:primal-int}
  \Xhat = \sum_{A \subseteq V} m_A
  \begin{bmatrix}
    1               & \incidvector{A}^{\transp}  \\[3pt]
    \incidvector{A} & \oprodsym{\incidvector{A}} \\
  \end{bmatrix}
  \text{ for some }
  m \colon \Powerset{V} \to \Integers_+.
  \tag{P\(\Integers\)}
\end{equation}

With these ``semidefinite integrality'' conditions in mind, we can
state a semidefinite analogue of \Cref{thm:LP-chain}.  To make the
theorems syntactically more similar, we shall adopt a more compact
notation for SDPs via linear maps: define \(\Acal \colon \Symhat{n}
\to \Reals^m\) by setting \([\Acal(\Xhat)]_i \coloneqq
\iprod{\Ahat_i}{\Xhat}\) for each \(i \in [m]\), so that
\(\Acal(\Xhat) \leq b\) is equivalent to \(\iprod{\Ahat_i}{\Xhat} \leq
b_i\) \(\forall i \in [m]\).  Then the adjoint \(\Acal^* \colon
\Reals^m \to \Symhat{n}\) satisfies \(\Acal^*(y) = \sum_{i=1}^m y_i
\Ahat_i\) for each \(y \in \Reals^m\).
\begin{theorem}
  \label{thm:1}
  If \(\Chat \in \Symhat{n}\) is a matrix, \(\Acal \colon \Symhat{\Vspc}
  \to \Reals^m\) is a linear map, and \(b \in \Reals^m\) is a vector,
  then
  \begin{align}
    \label{eq:chain-ISDP}
    &
    \sup\setst[\big]{
      \iprod{\Chat}{\Xhat}
    }{
      \Acal(\Xhat) \leq b,\,
      \text{\(\Xhat\) satisfies~\eqref{eq:primal-int}}
    }
    \tag{ISDP}
    \\
    \label{eq:chain-SDP}
    & \quad\quad\quad\quad
    \leq
    \sup\setst[\big]{
      \iprod{\Chat}{\Xhat}
    }{
      \Acal(\Xhat) \leq b,\,
      \Xhat \in \Psdhat{n}
    }
    \tag{SDP}
    \\
    \label{eq:chain-SDD}
    & \quad\quad\quad\quad\quad\quad
    \leq
    \inf\setst[\big]{
      \iprodt{b}{y}
    }{
      y \in \Reals_+^m,\,
      \Shat = \Acal^*(y) - \Chat \in \Psdhat{n}
    }
    \tag{SDD}
    \\
    \label{eq:chain-ISDD}
    & \quad\quad\quad\quad\quad\quad\quad\quad
    \leq
    \inf\setst[\big]{
      \iprodt{b}{y}
    }{
      y \in \Integers_+^m,\,
      \text{\(\Shat = \Acal^*(y) - \Chat\) satisfies \eqref{eq:dual-int}}
    },
    \tag{ISDD}
  \end{align}
  and the middle (second) inequality holds with equality if either one
  of~\eqref{eq:chain-SDP} and~\eqref{eq:chain-SDD} has a positive
  definite feasible solution and finite optimal value.
\end{theorem}

The equality in \Cref{thm:1} follows from the usual constraint
qualification for SDP, namely the fact that the SDP satisfies the
\emph{relaxed Slater condition}, i.e., the SDP has a positive definite
feasible solution; see, e.g., \cite[Theorem~1.1]{Carli13a}
or~\cite[Sec.~D.2.3]{BenTalN13a}.

We shall refer to~\eqref{eq:chain-ISDD} as the \emph{integer dual SDP}
of~\eqref{eq:chain-SDP}.  For convenience, we shall say that a
feasible solution~\((y,\Shat)\) for~\eqref{eq:chain-SDD} is
\emph{integral} if it is actually feasible in~\eqref{eq:chain-ISDD},
that~is, if \(y\) is integral and
\(\Shat\)~satisfies~\eqref{eq:dual-int}.  Integrality of~\(y\)
in~\eqref{eq:chain-ISDD} shows why it is important to use integral
matrices~\(\Ahat_i\).

Let us setup the integer dual SDP of the SDP
formulation~\eqref{eq:LP-as-SDP} of LPs.  If we require integrality
from feasible solutions of~\eqref{eq:LP-as-SDP-dual}, that~is, if~we
add the constraint~\eqref{eq:dual-int} and further constrain \(\eta\),
\(u\), \(y\), and \(Z\) to be integral, then~\eqref{eq:7} becomes
equivalent to
\begin{subequations}
  \label{eq:12}
  \begin{gather}
    \eta - \iprodt{b}{y} = \iprodt{\ones}{m},
    \\
    -u = -\sum_{A \subseteq V} m_A \incidvector{A},
    \\
    \Diag(2u + A^{\transp}y - c) = \sum_{A \subseteq V} m_A^{}
    \oprodsym{\incidvector{A}} + Z.
  \end{gather}
\end{subequations}
At each feasible solution we have \(Z \geq 0\), which implies that
\(\supp(m) \subseteq \tbinom{V}{1}\); we may always set
\(m_{\emptyset} \coloneqq 0\).  Thus, the integer dual SDP
of~\eqref{eq:LP-as-SDP} can be written as
\begin{subequations}
  \label{eq:13}
  \begin{alignat}{3}
    \text{Minimize}  \quad & \iprodt{\ones}{u} + \iprodt{b}{y}
    & \\[0pt]
    \text{subject to}\quad & A^{\transp}y + u \geq c,
    & \\[0pt]
                           & y \in \Integers_+^m,\, u \in \Integers_+^n,
    &
  \end{alignat}
\end{subequations}
assuming \(A\), \(b\), and~\(c\) to be integral.  Hence, every
feasible solution \(y\) for~\eqref{eq:chain-ILD} yields a feasible
solution for~\eqref{eq:12} with the same objective value by setting
\(u \coloneqq 0\).  In~fact,
\begin{equation}
  \label{claim:LP-as-SDP-dual-int-equiv}
  \text{\eqref{eq:13} is equivalent to \eqref{eq:chain-ILD}
    from~\Cref{thm:LP-chain} when \eqref{eq:chain-ILP} is \( \sup\setst{
      \iprodt{c}{x} }{ Ax \leq b,\, 0 \leq x \leq \ones,\, x \in
      \Integers^n }\).}
\end{equation}

From our previous discussion after~\eqref{eq:LP-as-SDP-dual}, our new
notion of dual integrality passes the test of behaving nicely with
respect to ILPs.  Next we will see that it surpasses the rank-one
constraint by showing that it~yields the ``natural'' combinatorial
dual for the theta function.

Let \(G = (V,E)\) be a graph.  A subset \(U\) of~\(V\) is a
\emph{clique} in~\(G\) if \(G[U] = K_U\).  Denote the set of cliques
of~\(G\) by~\(\Kcal(G)\).  Let \(w \colon V \to \Integers\).  The
\emph{clique covering number} \(\overline{\chi}(G,w)\) of~\(G\) with
respect to~\(w\) is the optimal value of the optimization problem
\begin{subequations}
  \label{eq:clique-covering}
  \begin{alignat}{2}
    \text{Minimize} \quad & \iprodt{\ones}{m}
    & \\[0pt]
    \text{subject to}\quad & m \colon \Powerset{V} \to \Integers_+,
    & \\[0pt]
    & \supp(m) \subseteq \Kcal(G),
    & \\[0pt]
    & \sum_{\mathclap{K \in \Kcal(G)}} m_K \incidvector{K} \geq w.
  \end{alignat}
\end{subequations}
Any feasible solution of~\eqref{eq:clique-covering} is a \emph{clique
  cover} of~\(G\) with respect to~\(w\).  We now show that the integer
dual SDPs for each of the SDP formulations~\eqref{eq:theta-SDP},
\eqref{eq:theta'-SDP}, and~\eqref{eq:theta+-SDP} are essentially
extended formulations for the clique covering number
\(\overline{\chi}(G,w)\):
\begin{proposition}
  \label{prop:theta-int-dual}
  Let \(G = (V,E)\) be a graph, and let \(w \colon V \to \Integers\).
  Then
  \begin{enumerate}[(i)]
  \item\label{prop:theta-int-dual-cover-to-int} if \(m \colon
    \Powerset{V} \to \Integers_+\) is a clique cover of~\(G\) with
    respect to~\(w\), then there exists an integral dual solution
    \((\Shat,\eta,u,y,z)\) for~\eqref{eq:SDP-theta'-dual} such that
    \eqref{eq:dual-int} holds for~\(\Shat\) and \(m \colon
    \Powerset{V} \to \Integers_+\), \(\eta = \iprodt{\ones}{m}\), and
    \(y \in \Reals_+^E \oplus 0\).
  \item\label{prop:theta-int-dual-int-to-cover} if
    \((\Shat,\eta,u,y,z)\) is an integral dual solution
    for~\eqref{eq:SDP-theta'-dual} and \eqref{eq:dual-int} holds
    for~\(\Shat\) and \(m \colon \Powerset{V} \to \Integers_+\), then
    \(\eta = \iprodt{\ones}{m}\) and \(m\) is a clique cover of~\(G\)
    with respect to~\(w\).
  \end{enumerate}
\end{proposition}
\begin{proof}
  To restrict ourselves to integral dual solutions
  for~\eqref{eq:SDP-theta'-dual}, we (i)~require the
  dual~slack~\(\Shat\) to satisfy~\eqref{eq:dual-int}, and
  (ii)~require \(\eta\), \(u\), \(y\), and \(z\) to be integral.  In
  this case, \eqref{eq:SDP-theta'-dual-eq} can be rewritten as
  \begin{gather}
    \eta = \iprodt{\ones}{m},
    \notag
    \\
    u = \sum_{A \subseteq V} m_A \incidvector{A},
    \notag
    \\
    \Diag(2u-z-w) + \sum_{ij \in \tbinom{V}{2}} 2y_{ij}
    \Symmetrize(\oprod{e_i}{e_j}) = \sum_{A \subseteq V} m_A^{} \oprodsym{\incidvector{A}}.
    \label{eq:14}
  \end{gather}

  Applying \(\diag\) to both sides of~\eqref{eq:14} yields \(2u-z-w
  = \sum_{A \subseteq V} m_A\incidvector{A} = u\).  Let \(i,j \in V\)
  be distinct.  The \(ij\)th entry of~\eqref{eq:14} is
  \begin{equation*}
    y_{ij} = \iprodt{\incidvector{\tbinom{V}{ij\subseteq}}}{m}.
  \end{equation*}
  Hence, the integer dual SDP of~\eqref{eq:theta'-SDP} can be written
  as
  \begin{subequations}
    \label{eq:SDP-theta'-dual-clean}
    \begin{alignat}{3}
      \text{Minimize} \quad & \eta & \\[0pt]
      \text{subject to}\quad &
      m \colon \Powerset{V} \to \Integers_+, & \\[0pt]
      & \eta = \iprodt{\ones}{m}, & \\[0pt]
      & u = \sum_{A \subseteq V} m_A \incidvector{A}, & \\[0pt]
      & u = w + z, & \\[0pt]
      \label{eq:SDP-theta'-dual-clean-y}
      & y_{ij} = \iprodt{\incidvector{\tbinom{V}{ij\subseteq}}}{m} \qquad
      \forall ij \in \tbinom{V}{2}, \\[0pt]
      \label{eq:SDP-theta'-dual-clean-Shat}
      & \Shat = \sum_{A \subseteq V} m_A
      \begin{bmatrix}
        1                & -\incidvector{A}^{\transp} \\[3pt]
        -\incidvector{A} & \oprodsym{\incidvector{A}} \\
      \end{bmatrix},
      \\
      &
      \Shat \in \Psdhat{\Vspc},\,
      \eta \in \Integers,\,
      u \in \Integers^V,\,
      z \in \Integers_+^V,\,
      y \in \Integers^E \oplus -\Integers_+^{\overline{E}}. &
    \end{alignat}
  \end{subequations}

  We may now prove the result.  We start
  with~\eqref{prop:theta-int-dual-cover-to-int}.  Suppose \(m \colon
  \Powerset{V} \to \Integers_+\) is a clique cover of~\(G\) with
  respect to~\(w\).  Set \(u \coloneqq \sum_{A \subseteq V} m_A
  \incidvector{A}\), \(z \coloneqq u - w \geq 0\), and \(\eta
  \coloneqq \iprodt{\ones}{m}\).  Define \(y\) and~\(\Shat\) as
  in~\eqref{eq:SDP-theta'-dual-clean-y}
  and~\eqref{eq:SDP-theta'-dual-clean-Shat}, respectively.  Since
  \(\supp(m) \subseteq \Kcal(G)\), we get \(y \in \Integers_+^E \oplus
  0\).  Hence, \((\Shat,\eta,u,y,z)\) is feasible
  in~\eqref{eq:SDP-theta'-dual-clean} and satisfies the desired
  properties in~\eqref{prop:theta-int-dual-cover-to-int}.

  For~\eqref{prop:theta-int-dual-int-to-cover}, let
  \((\Shat,\eta,u,y,z)\) be feasible
  in~\eqref{eq:SDP-theta'-dual-clean}.  If \(ij \in \overline{E}\),
  then \(y_{ij} \leq 0\) together
  with~\eqref{eq:SDP-theta'-dual-clean-y} yield \(m_A = 0\) for each
  \(A \subseteq V\) such that \(i,j \in V\).  Hence, \(m_A > 0\) and
  \(i,j \in A \subseteq V\) imply \(ij \in E\), i.e., \(\supp(m)
  \subseteq \Kcal(G)\), whence \(m\) is a clique cover of~\(G\).  This
  proves~\eqref{prop:theta-int-dual-int-to-cover}.
\end{proof}

Note that the result above is stated in a way to make it clear that
the integer dual SDPs of~\(\theta\), \(\theta'\), and~\(\theta^+\) are
all equivalent to the clique covering problem.

We have just seen that, not only the integer dual SDP has a feasible
solution for every graph, but it is actually equivalent to a natural
combinatorial optimization problem.  In~fact, the clique covering
problem is the \emph{right} dual problem for the maximum stable set
problem at least for the very rich class of perfect graphs; see, e.g.,
\cite[Ch.~67]{Schrijver03a}.  Recall that a graph \(G = (V,E)\) is
perfect if \(\omega(G[U]) = \chi(G[U])\) for each \(U \subseteq V\).

One may contend that the integrality constraints~\eqref{eq:primal-int}
and~\eqref{eq:dual-int} are not quite natural, and they depend
unnecessarily on having~\eqref{eq:SDP-01} as part of the constraints.
Note, however, that this arises from the choice of the
embedding~\eqref{eq:3}; the same objection might as well be raised for
ILPs, which have the arbitrary (though~intuitive) embedding using
incidence vectors.  That~is, the integrality conditions for ILP suffer
from the same drawbacks arising from the dependence on the embedding.
Other common drawbacks are that integrality constraints are not (and
probably cannot be) scaling invariant nor coordinate-free. The latter
drawbacks make it very hard to define a general integrality notion for
general convex relaxations; we discuss these issues
in~\Cref{sec:conclusion}.

Now that we have a sensible notion of integrality for the dual SDP, we
go back to the chain from~\Cref{thm:1}.  Motivated by the notion of
total dual integrality that was so powerful for proving equality
throughout in the chain from~\Cref{thm:LP-chain}, and which was based
on \Cref{thm:Edmonds-Giles} and~\Cref{cor:Hoffman}, we shall prove a
generalized version of the latter \namecref{cor:Hoffman} in the next
section.

\section{Integrality in Convex Relaxations}
\label{sec:int-convex}

In this section, we generalize \Cref{cor:Hoffman} to compact convex
sets.  This shall motivate the definition of total dual integrality
for SDPs in the next section.  Following~\cite{Rockafellar97a}, we
denote the \emph{support function} of a convex set \(\Cscr \subseteq
\Reals^n\) by
\begin{equation}
  \label{def:suppf}
  \suppf{\Cscr}{w} \coloneqq \sup_{x \in \Cscr} \iprod{w}{x}
  \in [-\infty,+\infty]
  \qquad
  \forall w \in \Reals^n.
\end{equation}
\begin{theorem}
  \label{rational-approx}
  If \(\Cscr \subseteq \Reals^n\) is a compact convex set, then
  \(\Cscr = \setst{x \in \Reals^n}{\iprodt{w}{x} \leq \suppf{\Cscr}{w}
    \,\forall w \in \Integers^n}\).
\end{theorem}

\begin{proof}
  We may assume that \(\Cscr \neq \emptyset\).  The inclusion
  `\(\subseteq\)' is obvious.  For the reverse inclusion, we start by
  noting that the RHS is equal to \hbox{\(\Cscr' \coloneqq \setst{x
      \in \Reals^n}{\iprodt{w}{x} \leq \suppf{\Cscr}{w} \,\forall w
      \in \Rationals^n}\)} by positive homogeneity
  of~\(\suppf{\Cscr}{\cdot}\).  Let \(\xbar \in \Cscr'\).  Let \(\wbar
  \in \Reals^n\), and let \((w_k)_{k \in \Naturals}\) be a sequence
  in~\(\Rationals^n\) converging to~\(\wbar\).  Then
  \(\iprodt{w_k}{\xbar} \leq \suppf{\Cscr}{w_k}\) for every \(k \in
  \Naturals\), which in the limit yields \hbox{\(\iprodt{\wbar}{\xbar}
    \leq \suppf{\Cscr}{\wbar}\)} by the (Lipschitz) continuity of the
  support function (apply Corollary 13.3.3 of~\cite{Rockafellar97a} to
  the function \(\suppf{\Cscr}{\cdot}\), where \(\Cscr\) is a compact
  convex set).  Hence \(\Cscr' \subseteq \setst{x \in
    \Reals^n}{\iprodt{w}{x} \leq \suppf{\Cscr}{w}\,\forall w \in
    \Reals^n} = \Cscr\), where the latter equation follows from
  Theorem 13.1 of~\cite{Rockafellar97a}.
\end{proof}

Note that the obvious generalization of~\Cref{rational-approx} to
unbounded convex set is false, even when restricted to
polyhedral~\(\Cscr\).  Consider, for~instance as~\(\Cscr\), any closed
halfspace with a normal vector containing both rational and irrational
entries.

Next we connect to the Gomory-Chvátal closure.  Let \(\Cscr \subseteq
\Reals^n\) be a convex set.  The \emph{Gomory-Chvátal closure}
of~\(\Cscr\) is
\begin{equation}
  \label{def:Gomory-Chvatal}
  \GomoryChvatal(\Cscr)
  \coloneqq
  \setst*{
    x \in \Reals^n
  }{
    \iprodt{w}{x} \leq \floor*{\suppf{\Cscr}{w}}
    \,\forall w \in \Integers^n
  }.
\end{equation}
The \emph{integer hull} of~\(\Cscr\) is
\begin{equation}
  \label{def:integer-hull}
  \Cscr_I \coloneqq \conv(\Cscr \cap \Integers^n).
\end{equation}

\begin{theorem}[\cite{Schrijver80a}]
  \label{cg-converges}
  If \(\Cscr \subseteq \Reals^n\) is a bounded convex set, then
  \(\GomoryChvatal^k(\Cscr) = \Cscr_I\) for some natural \(k \geq 1\).
\end{theorem}

We now generalize \Cref{cor:Hoffman}
(see~\cite{BraunP14a,DadushDV14a,DadushDV11a} for recent
generalizations in similar directions):
\begin{corollary}
  \label{cor:3.1}
  If \(\Cscr \subseteq \Reals^n\) is a nonempty compact convex set,
  then \(\Cscr = \Cscr_I\) if and only if \(\suppf{\Cscr}{w} \in
  \Integers\) for every \(w \in \Integers^n\).
\end{corollary}
\begin{proof}
  Necessity is clear.  For sufficiency, note that \(\Cscr = \setst{x
    \in \Reals^n}{\iprodt{w}{x} \leq
    \left\lfloor\suppf{\Cscr}{w}\right\rfloor\,\forall w \in
    \Integers^n} = \GomoryChvatal(\Cscr)\) by~\Cref{rational-approx}.
  Hence, \(\GomoryChvatal^k(\Cscr) = \Cscr\) for every \(k \geq 1\),
  so \(\Cscr = \Cscr_I\) by \Cref{cg-converges}.
\end{proof}

Characterizations of exactness of convex relaxations for sets of
integer points can naturally involve (convex) geometry in general,
boundary structure of convex sets in particular (including polyhedral
combinatorics), diophantine equations (number theory), and convex
analysis and optimization. Next, we summarize some of the consequences
of our geometric characterization (Corollary 10) of exactness for
convex relaxations of integral polytopes.  The next theorem,
well-known in the special case of LP relaxations, provides equivalent
characterizations of integrality in terms of the facial structure of
the convex relaxation, optimum values of linear functions over the
relaxation, optimal solutions of the linear optimization problems over
the relaxation, diophantine equations, and gauge functions in convex
optimization and analysis.

A convex subset \(\Fscr\) of a convex set \(\Cscr\) is a \emph{face}
of \(\Cscr\) if, for every \(x,y \in \Cscr\) such that the open line
segment \((x,y) \coloneqq \setst{\lambda x + (1-\lambda)y}{\lambda \in
  (0,1)}\) between~\(x\) and~\(y\) meets~\(\Fscr\), we have \(x,y \in
\Fscr\).  A nonempty face of \(\Cscr\) which does not contain another
nonempty face of \(\Cscr\) is a \emph{minimal face} of \(\Cscr\).  If
\(w \in \Reals^n \setminus \set{0}\) and \(\beta \in \Reals\), we say
that \(\Hscr \coloneqq \setst{x \in \Reals^n}{\iprodt{w}{x} \leq
  \beta}\) is a \emph{supporting halfspace} of~\(\Cscr\) if \(\Cscr
\subseteq \Hscr\); in this case we also say that \(\setst{x \in
  \Reals^n}{\iprodt{w}{x} = \beta}\) is a \emph{supporting hyperplane}
of~\(\Cscr\).  The intersection of~\(\Cscr\) with any of its
supporting hyperplanes is a face of~\(\Cscr\); such faces are
\emph{exposed}.

\begin{theorem}
  \label{thm:3.main}
  Let \(\Cscr\) be a nonempty compact convex set in
  \(\Reals^n\). Then, the following are equivalent:
  \smallskip
  \begin{enumerate}[(i)]
    \begin{minipage}{\textwidth}
    \item\label{item:1} \(\Cscr = \Cscr_I\);
    \item\label{item:2} every nonempty face of \(\Cscr\) contains an
      integral point;
    \item\label{item:3} every minimal face of \(\Cscr\) contains an
      integral point;
    \item\label{item:4} for every \(w \in \Reals^n\), we have that \(
      \max\setst{\iprod{w}{x}}{x \in \Cscr} \) is attained by an
      integral point;
    \item\label{item:5} for every \(w \in \Integers^n\), we have \(
      \max\setst{\iprod{w}{x}}{x \in \Cscr} \in \Integers \);
    \item\label{item:6} every rational supporting hyperplane for
      \(\Cscr\) contains integral points;
    \item\label{item:7} for each \(x_0 \in \Cscr\) and for each \(w
      \in \Integers^n\), we have \(\iprod{w}{x_0} + \inf\setst{\eta
        \in \Reals_{++}}{\tfrac{1}{\eta}w \in (\Cscr-x_0)^{\circ}} \in
      \Integers\);
    \item\label{item:8} there exists \(x_0 \in \Cscr\) such that, for
      each \(w \in \Integers^n\), \(\iprod{w}{x_0} + \inf\setst{\eta
        \in \Reals_{++}}{\tfrac{1}{\eta}w \in (\Cscr-x_0)^{\circ}} \in
      \Integers\).
    \end{minipage}
  \end{enumerate}
\end{theorem}

\begin{proof}
  \(\text{\eqref{item:1}} \Rightarrow \text{\eqref{item:2}}\): Since
  \(\Cscr\) is compact, it is bounded.  Therefore, \(\Cscr=\Cscr_I\)
  implies that \(\Cscr\) is a polytope.  Every nonempty face of
  \(\Cscr\) contains an extreme point of \(\Cscr\) and every extreme
  point of \(\Cscr=\Cscr_I\) is integral.

  \(\text{\eqref{item:2}} \Rightarrow \text{\eqref{item:3}}\):
  Immediate.

  \(\text{\eqref{item:3}} \Rightarrow \text{\eqref{item:4}}\): Suppose
  every minimal face of \(\Cscr\) contains an integral point.  Let \(w
  \in \Reals^n\).  Then, since \(\Cscr\) is nonempty, compact and
  convex,
  \[
  \argmax_{x \in \Cscr}{\iprod{w}{x}} \eqqcolon \Fscr
  \]
  is a nonempty (exposed) face of \(\Cscr\).  Every minimal face
  contained in \(\Fscr\) contains an integral point
  (by~part~\eqref{item:3}); hence, \(\Fscr\) contains an integral point.

  \(\text{\eqref{item:4}} \Rightarrow \text{\eqref{item:5}}\): Suppose
  \(\Cscr\) satisfies~\eqref{item:4}.  Let \(w \in \Integers^n\).
  Then, by~\eqref{item:4}, there exists \(\bar{x} \in \Cscr \cap
  \Integers^n\) such that
  \[
  \max_{x \in \Cscr}{\iprod{w}{x}} = \iprod{w}{\bar{x}}.
  \]
  Since \(w\) and~\(\bar{x}\) are integral, it follows that \(\max_{x
    \in \Cscr}{\iprod{w}{x}} \in \Integers\).

  \(\text{\eqref{item:5}} \Rightarrow \text{\eqref{item:6}}\): Suppose
  \(\Cscr\) has the property~\eqref{item:5}.  Let \(w \in
  \Rationals^n\).  Define
  \[
  \Fscr \coloneqq \argmax_{x \in \Cscr}{\iprod{w}{x}}.
  \]
  Let \(\mu\) be a positive rational such that \(\mu w \in
  \Integers^n\) and \(\gcd(\mu w_1, \dotsc, \mu w_n)=1\).  Then,
  \(\argmax_{x \in \Cscr}{\iprod{\mu w}{x}} = \Fscr\).  By
  property~\eqref{item:5}, \(\beta \coloneqq \max_{x \in
    \Cscr}{\iprod{\mu w}{x}} \in \Integers\).  Since
  \[
  \setst{x \in \Integers^n}{\iprod{\mu w}{x} = \beta} \neq \emptyset
  \iff
  \text{\(\gcd(\mu w_1, \dotsc, \mu w_n)\) divides \(\beta\)},
  \]
  and we have \(\gcd(\mu w_1, \dotsc, \mu w_n) = 1\), we are done.

  \(\text{\eqref{item:6}} \iff \text{\eqref{item:1}}\): Suppose
  \(\Cscr\) has property \eqref{item:6}.  Then, for every \(w \in
  \Integers^n\), \(\suppf{\Cscr}{w} \in \Integers\).  Therefore, by
  \Cref{cor:3.1}, \(\Cscr= \Cscr_I\).  The converse also follows from
  \Cref{cor:3.1}.

  \(\text{\eqref{item:5}} \iff \text{\eqref{item:7}} \iff \text{\eqref{item:8}}\): Let \(x_0 \in
  \Cscr\) and \(w \in \Integers^n\).  Set \(\Cscrtilde \coloneqq \Cscr
  - x_0\).  Then
  \begin{equation*}
    \suppf{\Cscr}{w}
    = \iprod{w}{x_0} + \suppf{\Cscrtilde}{w}
    = \iprod{w}{x_0}
    + \min\setst[\big]{
      \eta \in \Reals_+
    }{
      \iprod{w}{x} \leq \eta\,\forall x \in \Cscrtilde
    },
  \end{equation*}
  where in the last equation we use the fact that \(0 \in \Cscrtilde\)
  to add the constraint \(\eta \in \Reals_+\).  Finally, note that
  \begin{equation*}
    \min\setst[\big]{
      \eta \in \Reals_+
    }{
      \iprod{w}{x} \leq \eta\,\forall x \in \Cscrtilde
    }
    =
    \inf\setst[\big]{
      \eta \in \Reals_{++}
    }{
      \tfrac{1}{\eta} w \in \Cscrtilde^{\circ}
    }.\qedhere
  \end{equation*}
\end{proof}

In the quite common case that \(0 \in \Cscr\), \Cref{thm:3.main} shows
that \(\Cscr_I = \Cscr\) if and only if, for each \(w \in
\Integers^n\), we~have \(\inf\setst{\eta \in
  \Reals_{++}}{\tfrac{1}{\eta}w \in \Cscr^{\circ}} \in \Integers\).

Just as \Cref{thm:Edmonds-Giles} motivates the definition of total
dual integrality for LP formulations, one may use \Cref{cor:3.1} to
define total dual integrality more generally.  In the next section, we
shall define it for SDP formulations.

\section{Total Dual Integrality for SDPs}
\label{sec:sdp-tdi}

Before we define a semidefinite notion of total dual integrality, we
shall recall a few basic facts about the corresponding theory for
polyhedra.  Let \(A \in \Reals^{m \times n}\) be a matrix, and let \(b
\in \Reals^m\).  We say that the system \(Ax \leq b\) is
\emph{rational} if the entries of~\(A\) and~\(b\) are rational.  The
rational system of linear inequalities \(Ax \leq b\) is \emph{totally
  dual integral} (\emph{TDI}) if, for each \(c \in \Integers^m\), the
LP \(\min\setst{\iprodt{b}{y}}{A^{\transp}y = c,\, y \geq 0}\) dual to
\(\max\setst{\iprodt{c}{x}}{Ax \leq b}\) has an integral optimal
solution if its optimal value is finite.  If \(Ax \leq b\) is TDI and
\(b\) is integral, then the polyhedron \(P \coloneqq \setst{x \in
  \Reals^n}{Ax \leq b}\) is integral by~\Cref{thm:Edmonds-Giles}.  It
is important to emphasize that total dual integrality is an
\emph{algebraic} notion, rather than a \emph{geometric} one: it is not
the geometric object~\(P\) that is TDI, but rather the defining system
\(Ax \leq b\), which is not uniquely determined by~\(P\).  This
subtlety leads to some odd consequences, as we describe next.

A polyhedron \(P \subseteq \Reals^n\) is \emph{rational} if it is
determined by a rational system of linear inequalities.  It is
well~known~\cite[Theorem~22.6]{Schrijver86a} that every rational
polyhedron~\(P\) is defined by a TDI system \(Ax \leq b\) with
\(A\)~integral, and if \(P\) is integral, then~\(b\) may be chosen
integral.  This allows one to prove the odd fact that, for~every
rational system \(Ax \leq b\), there is a positive integer~\(t\) such
that the system \((\tfrac{1}{t}A)x \leq \tfrac{1}{t}b\) is TDI.

Next we move on to define a notion of total dual integrality for SDP
formulations.  We want to define when the defining system
\(\Acal(\Xhat) \leq b\), \(\Xhat \succeq 0\) for~\eqref{eq:chain-SDP}
is~TDI, but there is a further complication.  We may not need the dual
SDP to have an ``integral solution'' for \emph{every} integral
objective function \(\Xhat \mapsto \iprod{\Chat}{\Xhat}\).  As~the
formulation~\eqref{eq:theta-SDP} shows, for the
Lovász~\(\theta\)~function we are only interested in objective
functions of the form \(\Xhat \mapsto \iprod{\Diag(0\oplus
  w)}{\Xhat}\), perhaps with \(w \in \Reals^n\) integral.  The same
remark can be made about the diagonal embedding~\eqref{eq:LP-as-SDP}
of LPs as SDPs.  In these cases, one is interested only in the
\emph{diagonal} part of the variable~\(\Xhat\), and the lifting \(w
\mapsto \Diag(0 \oplus w)\) embeds in matrix space only the objective
functions that matter to us.  Note that this arises from the fact that
we are essentially dealing with extended (lifted) formulations.
However, when we look at the MaxCut~SDP in~\Cref{sec:maxcut}, we~shall
only be interested in objective functions of the form \(\Xhat \mapsto
\iprod{0 \oplus \Laplacian{G}(w)}{\Xhat}\), where \(\Laplacian{G}(w)
\in \Sym{V}\) is a weighted Laplacian matrix of the input graph~\(G\)
on vertex~set~\(V\), to be defined later; as~before, \(\Xhat \in
\Psdhat{\Vspc}\) is the variable.  In this case, one might argue the
we are only interested in the \emph{off-diagonal} (!)  entries of the
variable~\(\Xhat\).  Thus, when defining semidefinite TDIness,
we~shall need to refer to which objective functions (that~is, which
projection of the feasible region) we care about.  (This notion of
TDIness coupled with extended formulations already leads to an
interesting generalization of TDIness in the polyhedral case, as we
discuss in~\Cref{sec:conclusion}.)

We may now define a semidefinite notion of total dual integrality.
Below, the map~\(\Lcal\) is a lifting map, such as \(w \mapsto \Diag(0
\oplus w)\) and \(w \mapsto 0 \oplus \Laplacian{G}(w)\) from above.  The
corresponding projection, which will be the adjoint \(\Lcal^*\) of the
lifting~\(\Lcal\), will appear in~\Cref{thm:tdi-projection} below.

\begin{definition}
  \label{def:TDI}
  Let \(\Lcal \colon \Reals^k \to \Symhat{n}\) be a linear map.  The
  system \(\Acal(\Xhat) \leq b\), \(\Xhat \succeq 0\) is \emph{totally
    dual integral (TDI) through~\(\Lcal\)} if, for every integral \(c \in
  \Integers^k\), the SDP dual to
  \(\sup\setst{\iprod{\Lcal(c)}{\Xhat}}{\Acal(\Xhat) \leq b,\, \Xhat
    \succeq 0}\) has an integral optimal solution whenever it has an
  optimal solution.
\end{definition}

Note that, for~convenience, we use the term ``TDI'' to refer to two
separate notions, one for linear inequality systems of the form \(Ax
\leq b\), and another one for semidefinite systems of the form
\(\Acal(\Xhat) \leq b\), \(\Xhat \succeq 0\); the~context shall make
it clear to which notion we are referring.

\begin{theorem}
  \label{thm:tdi-projection}
  Let \(\Acal(\Xhat) \leq b\), \(\Xhat \succeq 0\) be totally dual
  integral through a linear map \(\Lcal \colon \Reals^k \to \Symhat{n}\).
  Set \(\Cscrhat \coloneqq \setst{\Xhat \in \Psdhat{n}}{\Acal(\Xhat)
    \leq b}\) and \(\Cscr \coloneqq \Lcal^*(\Cscrhat) \subseteq
  \Reals^k\).  If \(b\) is integral, \(\Cscr\) is compact, and
  \(\Cscrhat\) has a positive definite matrix, then \(\Cscr =
  \Cscr_I\).
\end{theorem}
\begin{proof}
  Let \(w \in \Integers^k\).  Then
  \begin{equation}
    \label{eq:15}
    \suppf{\Cscr}{w}
    = \max_{\Xhat \in \,\Cscrhat} \iprod{w}{\Lcal^*(\Xhat)}
    = \max_{\Xhat \in \,\Cscrhat} \iprod{\Lcal(w)}{\Xhat}.
  \end{equation}
  The latter SDP satisfies the relaxed Slater condition by assumption
  and its optimal value is finite and attained by compactness
  of~\(\Cscr\).  By SDP~Strong~Duality, the dual SDP has an optimal
  solution.  Since \(\Acal(\Xhat) \leq b\), \(\Xhat \succeq 0\) is TDI
  through~\(\Lcal\), the dual SDP has an integral optimal solution
  \((y^*,\Shat^*)\).  Hence, \(\suppf{\Cscr}{w} = \iprodt{b}{y^*}\)
  and so \(\suppf{\Cscr}{w} \in \Integers\), since \(b\) is integral.
  It~follows from~\Cref{cor:3.1} that \(\Cscr = \Cscr_I\).
\end{proof}

We have established that total dual integrality is sufficient for
exact (primal) representations.  We next describe conditions under
which the chain of inequalities in~\Cref{thm:1} holds with equality
throughout, thus completing our discussion in~\Cref{sec:intro}
regarding equality throughout in~\Cref{thm:LP-chain}.

Again there is a more involved setup due to our choice of
embedding~\eqref{eq:3}.  Let \(\Cscr \subseteq [0,1]^k\) be a
convex~set.  Let \(\Lcal \colon \Reals^k \to \Symhat{n}\) be a linear
map, and let \(\Cscrhat \subseteq \Symhat{n}\).  We say that
\(\Cscrhat\) is a \emph{rank-one embedding} of~\(\Cscr_I\)
via~\(\Lcal\) if, for~each \(\xbar \in \set{0,1}^k\) there exists
\(\Xhat \in \Cscrhat\) such that \(\xbar = \Lcal^*(\Xhat)\) and
\(\Xhat\) has the form~\eqref{eq:3} for some \(U \subseteq V \coloneqq
[n]\).  One may think of~\(\Cscrhat\) as a convex set in (lifted)
matrix space, e.g., the feasible region of an~SDP, described
algebraically by a linear system \(\Acal(\Xhat) \leq b\), \(\Xhat
\succeq 0\) that includes~\eqref{eq:SDP-01}.  Then to have the
(lifted) rank-constrained SDP formulation
\(\sup\setst{\iprod{\Lcal(w)}{\Xhat}}{\Xhat \in \Cscrhat,\,
  \rank(\Xhat) = 1}\) be a correct relaxation for the combinatorial
optimization problem \(\max\setst{\iprodt{w}{x}}{x \in \Cscr \cap
  \set{0,1}^k}\) requires the conditions for \(\Cscrhat\) to be a
rank-one embedding of~\(\Cscr_I\).

In the case where \(\Lcal \colon w \in \Reals^V \mapsto 0 \oplus
\Diag(w)\) and \(\Cscr \subseteq [0,1]^V\), to say that the set
\(\Cscrhat \subseteq \Symhat{n}\) defined by a system \(\Acal(\Xhat)
\leq b\), \(\Xhat \succeq 0\) is a rank-one embedding of \(\Cscr_I\)
via~\(\Lcal\) requires that, for each \(\xbar \in \Cscr \cap
\set{0,1}^V\), we have
\[
\Acal\paren[\bigg]{
  \begin{bmatrix}
    1     & \xbar^{\transp}  \\[3pt]
    \xbar & \oprodsym{\xbar} \\
  \end{bmatrix}
} \leq b.
\]

\begin{theorem}
  \label{thm:eq-throughout}
  Let \(\Acal(\Xhat) \leq b\), \(\Xhat \succeq 0\) be totally dual
  integral through a linear map \(\Lcal \colon \Reals^k \to
  \Symhat{n}\) such that \(b\)~is~integral.  Suppose that \(\Cscrhat
  \coloneqq \setst{\Xhat \in \Psdhat{n}}{\Acal(\Xhat) \leq b}\) has a
  positive definite matrix and that \(\Cscr \coloneqq
  \Lcal^*(\Cscrhat) \subseteq [0,1]^k\) is compact.  If \(\Cscrhat\)
  is a rank-one embedding of~\(\Cscr_I\) via~\(\Lcal\), then for every
  \(w \in \Integers^k\), equality holds throughout in the chain of
  inequalities from~\Cref{thm:1} for \(\Chat \coloneqq \Lcal(w)\), all
  optimum values are equal to
  \begin{equation}
    \label{eq:28}
    \max\setst{\iprodt{w}{x}}{x \in \Cscr_I},
  \end{equation}
  and all suprema and infima are attained.
\end{theorem}
\begin{proof}
  Fix \(w \in \Integers^k\) and set \(\Chat \coloneqq \Lcal(w)\)
  throughout the proof.  Note that the optimal value
  of~\eqref{eq:chain-SDP} is bounded above, since each \(\Xhat \in
  \Cscrhat\) has objective value \(\iprod{\Lcal(w)}{\Xhat} =
  \iprod{w}{\Lcal^*(\Xhat)} \leq \suppf{\Cscr}{w} < \infty\) by
  compactness.  Since the relaxed Slater condition holds by
  assumption, SDP Strong Duality shows that \eqref{eq:chain-SDD} has
  an optimal solution, and hence is feasible.  Together with the TDI
  assumption, this shows that \eqref{eq:chain-SDP},
  \eqref{eq:chain-SDD}, and \eqref{eq:chain-ISDD} have the same
  optimal values and the latter two are attained.

  It~remains to prove that~\eqref{eq:chain-SDP},
  \eqref{eq:chain-ISDP}, and~\eqref{eq:28} have the same optimal
  values and are attained.  Let \(\xbar\) be an optimal solution for
  \(\max\setst{\iprodt{w}{x}}{x \in \Cscr \cap \set{0,1}^k}\).  Then
  there exists \(\Xbar \in \Cscrhat\) that
  satisfies~\eqref{eq:primal-int} such that \(\xbar =
  \Lcal^*(\Xbar)\).  Then the optimal value of~\eqref{eq:28} is
  \(\iprodt{w}{\xbar} = \iprod{w}{\Lcal^*(\Xbar)} =
  \iprod{\Chat}{\Xbar}\), which is upper bounded by the optimal value
  of~\eqref{eq:chain-ISDP}.  On the other hand, as shown above the
  optimal value of~\eqref{eq:chain-SDP} is upper bounded by
  \(\suppf{\Cscr}{w} = \suppf{\Cscr_I}{w} = \iprodt{w}{\xbar}\) since
  \(\Cscr = \Cscr_I\) by \Cref{thm:tdi-projection}.  Hence, \(\xbar\)
  is optimal in~\eqref{eq:28}, and \(\Xbar\) is optimal
  in~\eqref{eq:chain-ISDP} and~\eqref{eq:chain-SDP}, all with the same
  objective values.
\end{proof}

Naturally, any other choice of (i) embedding in some lifted space and
(ii) integrality conditions would require an adaptation of the
definition of ``rank-one embedding'' of~\(\Cscr_I\) via a lifting map,
if only to ensure that the lifted representation~\(\Cscrhat\) is a
correct formulation of~\eqref{eq:28}.

The next result characterizes TDIness for the diagonal
embedding~\eqref{eq:LP-as-SDP} of LPs.  It shows that our notion of
semidefinite TDIness is the same as the polyhedral notion, given the
limitation in our model that only deals with binary variables:
\begin{theorem}
  \label{thm:SDP-TDI-LP}
  Let \(Ax \leq b\) be a rational system of linear inequalities.  The
  system defining~\eqref{eq:LP-as-SDP} is TDI through \(w \in \Reals^V
  \mapsto \Diag(0 \oplus w)\) if and only if the system \(Ax \leq b\),
  \(0 \leq x \leq \ones\) is TDI.
\end{theorem}
\begin{proof}
  Immediate from~\eqref{claim:LP-as-SDP-dual-int-equiv}.
\end{proof}

Together with~\Cref{thm:SDP-TDI-LP}, \Cref{thm:eq-throughout} yields a
richer version of equality throughout in the chain
from~\Cref{thm:LP-chain}, since it includes the LP case via the
diagonal embedding~\eqref{eq:LP-as-SDP} as well as other, lifted
formulations; see, e.g., \Cref{thm:theta-TDI} in the next section.
\Cref{thm:eq-throughout} yields further results when the lifting map
involves the Laplacian of a graph~\(G\), i.e., when~\(\Lcal\) has the
form \(w \mapsto 0 \oplus \Laplacian{G}(w)\) as discussed
before~\Cref{def:TDI}.  In this case, we leave it to the reader to
check exactly how the set \(\Cscr\) must be related to the (incidence~vectors~of) cuts of~\(G\).

\section{Integrality in the Theta Function Formulation}
\label{sec:theta-tdi}

In this section, we prove that the formulation~\eqref{eq:theta-SDP}
for the Lovász \(\theta\) function of a graph~\(G\) is TDI through the
appropriate lifting if and only if \(G\) is perfect.

Let \(G = (V,E)\) be a graph.  We say that \(G\) is \emph{perfect} if
\(\omega(G[U]) = \chi(G[U])\) for every \(U \subseteq V\).  For~each
\(w \colon V \to \Reals\), the \emph{weighted stability number}
\(\alpha(G,w)\) of~\(G\) with respect to~\(w\) is
\begin{equation}
  \label{eq:def-alpha}
  \alpha(G,w)
  \coloneqq
  \max\setst{
    \iprodt{w}{\incidvector{U}}
  }{
    \text{\(U \subseteq V\) stable}
  }.
\end{equation}

A subset~\(\Cscr\) of~\(\Reals_+^n\) is a \emph{convex corner}
if~\(\Cscr\) is a compact convex set with nonempty interior and such
that \(0 \leq y \leq x \in \Cscr\) implies \(y \in \Cscr\).  Associate
with each graph \(G = (V,E)\) the following convex corners:
\begin{gather*}
  \STAB(G) \coloneqq \conv\setst{\incidvector{U}}{\text{\(U \subseteq V\) stable}},
  \\
  \THbody'(G) \coloneqq \setst{
    \diag(\Xhat[V])
  }{
    \text{\(\Xhat\) feasible in~\eqref{eq:theta'-SDP}}
  },
  \\
  \THbody(G) \coloneqq \setst{
    \diag(\Xhat[V])
  }{
    \text{\(\Xhat\) feasible in~\eqref{eq:theta-SDP}}
  },
  \\
  \THbody^+(G) \coloneqq \setst{
    \diag(\Xhat[V])
  }{
    \text{\(\Xhat\) feasible in~\eqref{eq:theta+-SDP}}
  },
  \\
  \QSTAB(G) \coloneqq \setst{
    x \in \Reals_+^V
  }{
    \iprodt{\incidvector{K}}{x} \leq 1\,\forall K \in \Kcal(G)
  }.
\end{gather*}
A strong form of Lovász sandwich theorem~\cite{Lovasz79a} is
that
\begin{equation}
  \label{eq:strong-sandwich}
  \STAB(G)
  \subseteq \THbody'(G)
  \subseteq \THbody(G)
  \subseteq \THbody^+(G)
  \subseteq \QSTAB(G).
\end{equation}
The following result is well known; we include a sketch of its proof
for completeness.
\begin{theorem}
  \label{thm:perfect-TFAE}
  Let \(G\) be a graph.  The following are equivalent:
  \begin{enumerate}[(i)]
    \setlength{\itemsep}{3pt}
  \item \(G\) is perfect;
  \item \(\overline{G}\) is perfect;
  \item\label{item:perfect-TFAE-STAB-QSTAB} \(\STAB(G) = \QSTAB(G)\);
  \item the system \(x \geq 0\), \(\iprodt{\incidvector{K}}{x} \leq
    1\) \(\forall K \in \Kcal(G)\) defining~\(\QSTAB(G)\) is TDI;
  \item\label{item:perfect-TFAE-alpha-chi} \(\alpha(G,w) = \overline{\chi}(G,w)\) for each \(w \colon V
    \to \Integers\);
  \item\label{item:perfect-TFAE-TH} \(\THbody(G)\) is a polytope;
  \item\label{item:perfect-TFAE-TH'} \(\THbody'(G)\) is a polytope;
  \item\label{item:perfect-TFAE-TH+} \(\THbody^+(G)\) is a polytope.
  \end{enumerate}
\end{theorem}
\begin{proof}
  Most equivalences can be seen in~\cite[Ch.~9]{GroetschelLS93a},
  except for~\eqref{item:perfect-TFAE-TH'}
  and~\eqref{item:perfect-TFAE-TH+}, involving~\(\THbody'(G)\)
  and~\(\THbody^+(G)\).  It is clear
  that~\eqref{item:perfect-TFAE-STAB-QSTAB}
  and~\eqref{eq:strong-sandwich} imply
  both~\eqref{item:perfect-TFAE-TH'}
  and~\eqref{item:perfect-TFAE-TH+}.  When proving
  that~\eqref{item:perfect-TFAE-TH}
  implies~\eqref{item:perfect-TFAE-STAB-QSTAB},
  \cite[Cor~9.3.27]{GroetschelLS93a} relies on the facts that the
  antiblocker of~\(\THbody(G)\) is~\(\THbody(\overline{G})\) and that
  the nontrivial facets of~\(\THbody(G)\) are determined by the clique
  inequalities \(\iprodt{\incidvector{K}}{x} \leq 1\) for each \(K \in
  \Kcal(G)\).  It is well known that the antiblocker
  of~\(\THbody'(G)\) is~\(\THbody^+(\overline{G})\) and that the
  nontrivial facets of both~\(\THbody'(G)\) and~\(\THbody^+(G)\) are
  determined by the same clique inequalities above.  The~interested
  reader can find complete, unified proofs
  in~\cite[Theorem~24]{CarliT17a}.  These facts are sufficient to
  adapt the proof from~\cite[Cor.~9.3.27]{GroetschelLS93a} to show
  that each of~\eqref{item:perfect-TFAE-TH'}
  and~\eqref{item:perfect-TFAE-TH+}, separately,
  implies~\eqref{item:perfect-TFAE-STAB-QSTAB}.
\end{proof}

We can now characterize TDIness for~\(\theta\) via perfection.  We
comment in the proof below the modifications to obtain analogous
results for the formulations~\eqref{eq:theta'-SDP}
and~\eqref{eq:theta+-SDP}, of~\(\theta'\) and~\(\theta^+\),
respectively.
\begin{theorem}
  \label{thm:theta-TDI}
  Let \(G = (V,E)\) be a graph.  The defining system for the SDP
  formulation of Lovász~\(\theta\) function in~\eqref{eq:theta-SDP} is
  TDI through \(w \in \Reals^V \mapsto \Diag(0 \oplus w)\) if and only
  if \(G\)~is~perfect.
\end{theorem}
\begin{proof}
  We start with sufficiency.  Suppose \(G\) is perfect.  Let \(w
  \colon V \to \Integers\).  Let \(U \subseteq V\) be a stable set
  of~\(G\) such that \(\alpha(G,w) = \iprodt{w}{\incidvector{U}}\), so
  that \(\Xhat\) defined as in~\eqref{eq:3} is feasible
  in~\eqref{eq:theta-SDP} with objective value \(\alpha(G,w)\).  Then
  by item~\eqref{item:perfect-TFAE-alpha-chi}
  in~\Cref{thm:perfect-TFAE} there exists a clique cover \(m\)
  of~\(G\) with respect to~\(w\) such that \(\iprodt{\ones}{m} =
  \alpha(G,w)\).  Hence, \Cref{prop:theta-int-dual} shows that there
  is an integral dual solution \((\Shat,\eta,u,y,z)\) for the dual SDP
  of~\eqref{eq:theta-SDP} with objective value \(\eta =
  \iprodt{\ones}{m} = \alpha(G,w)\), which is the same as the
  objective value of~\(\Xhat\).  Hence, \((\Shat,\eta,u,y,z)\) is
  optimal for the dual SDP of~\eqref{eq:theta-SDP} by weak duality.
  Note in~fact that \Cref{prop:theta-int-dual} shows that
  \((\Shat,\eta,u,y,z)\) is an integer dual solution also for the dual
  SDPs of~\eqref{eq:theta'-SDP} and~\eqref{eq:theta+-SDP}.

  Now we move to necessity.  Suppose the defining system is TDI
  through \(\Diag(0 \oplus \cdot)\).  By~\Cref{thm:tdi-projection},
  it~follows that \(\THbody(G) = \THbody(G)_I\), hence \(\THbody(G)\)
  is a polytope and \(G\) is perfect by~\Cref{thm:perfect-TFAE}.  Note
  that the equivalences~\eqref{item:perfect-TFAE-TH'}
  and~\eqref{item:perfect-TFAE-TH+} in~\Cref{thm:perfect-TFAE} also
  show that the defining systems for~\(\theta'\) and~\(\theta^+\) can
  only be~TDI if~\(G\) is perfect.
\end{proof}

\section{Dual Integrality for the MaxCut SDP}
\label{sec:maxcut}

Let \(G = (V,E)\) be a graph.  A \emph{cut} in~\(G\) is a set of edges
of the form
\begin{equation}
  \label{def:cut}
  \delta(U)
  \coloneqq
  \setst{e \in E}{\card{e \cap U} = 1}
\end{equation}
for some \(U \subseteq V\) such that \(\emptyset \neq U \neq V\).  The
\emph{maximum cut problem} (or MaxCut problem) is to find, given a
graph \(G = (V,E)\) and \(w \colon E \to \Reals_+\), an optimal
solution for
\({
  \max\setst{
    \iprodt{w}{\incidvector{\delta(U)}}
  }
  {
    \emptyset \neq U \subsetneq V
  }
}\).
(We shall discuss nonnegativity of~\(w\) and related issues
in~\Cref{sec:maxcut-nonneg}.)  It is well known that, by using the
embedding \(U \in \Powerset{V} \mapsto \oprodsym{s_U} \in \Sym{V}\)
with \(s_U \coloneqq 2\incidvector{U}-\ones\), i.e.,
\begin{equation}
  \label{eq:plus-minus-embedding}
  (s_U)_i
  =
  (-1)^{[i \not\in U]}
  \qquad
  \forall i \in V,
\end{equation}
one may reformulate the MaxCut problem exactly by adding the
constraint ``\(\rank(Y) = 1\,\)'' to the SDP
\begin{equation}
  \label{eq:maxcut-sdp}
  \begin{array}[!h]{rll}
    \text{Maximize}   & \iprod{\tfrac{1}{4}\Laplacian{G}(w)}{Y} &                  \\[4pt]
    \text{subject to} & \iprod{\oprodsym{e_i}}{Y} = 1           & \forall i \in V, \\[4pt]
                      & Y \in \Psd{V};                          &                  \\
  \end{array}
\end{equation}
here, \(\Laplacian{G} \colon \Reals^E \to \Sym{V}\) is the \emph{Laplacian} of
the graph~\(G\), defined as
\begin{equation}
  \label{eq:def-Laplacian}
  \Laplacian{G}(w) \coloneqq \sum_{ij \in E} w_{ij} \oprodsym{(e_i-e_j)}
  \qquad
  \forall w \in \Reals^E.
\end{equation}
It~is not hard to check that
\(\qform{\Laplacian{G}(w)}{\incidvector{U}} = \tfrac{1}{4}
\qform{\Laplacian{G}(w)}{s_U} = \iprodt{w}{\incidvector{\delta(U)}}\)
for each \(U \subseteq V\), with \(s_U\) defined as
in~\eqref{eq:plus-minus-embedding}.  We call~\eqref{eq:maxcut-sdp} the
\emph{MaxCut SDP}.  It is one of the most famous~SDPs, since it was
used by Goemans and Williamson~\cite{GoemansW95a} in their seminal
approximation algorithm and its analysis.

We discuss the drawbacks of the rank-one constraint for the dual SDP
of~\eqref{eq:maxcut-sdp} in \Cref{sec:rank-maxcut}, and
in~\Cref{sec:maxcut-dual-int} we study the integer dual SDP for the
MaxCut SDP with objective functions of the form \(X \mapsto
\iprod{\tfrac{1}{4} \Laplacian{G}(w)}{X}\) for every \(w \in
\Reals_+^E\).

\subsection{Rank-One Constraint in Dual of the MaxCut SDP}
\label{sec:rank-maxcut}

In this section, we show that the dual of the MaxCut SDP has a
feasible solution with a rank-one slack only if the weight function on
the edges comes from a very restricted (though rather interesting)
class of weight functions.  Let \(G = (V,E)\) be a graph and let \(w
\colon E \to \Reals\).  The dual of the MaxCut
SDP~\eqref{eq:maxcut-sdp} is
\begin{equation}
  \label{eq:27}
  \begin{array}[!h]{rll}
    \text{Minimize}  \quad & \iprodt{\ones}{y}                           & \\[0pt]
    \text{subject to}\quad & S = \Diag(y) - \tfrac{1}{4}\Laplacian{G}(w) & \\[2pt]
                           & S \in \Psd{V},\, y \in \Reals^V.            &
  \end{array}
\end{equation}

\begin{proposition}
  \label{prop:maxcut-rankone}
  Let \(G = (V,E)\) be a graph without isolated vertices.  Let \(w
  \colon E \to \Reals \setminus \set{0}\).  If~\eqref{eq:27} has a
  feasible solution \((S,y)\) such that \(\rank(S) \leq 1\), then \(G
  = K_V\), and there exists \(u \colon V \to \Reals \setminus
  \set{0}\) such that \(w_{ij} = u_i u_j\) for each \(ij \in E\).
\end{proposition}
\begin{proof}
  Set \(L \coloneqq \Laplacian{G}(w)\).  Suppose there exists \(u \in
  \Reals^V\) such that \(S = \oprodsym{u}\).  Then, for each \(i \in
  V\), we have \(y_i - \tfrac{1}{4}L_{ii} = S_{ii} = u_i^2 \geq 0\);
  equality implies that \(S e_i = 0\).  Since \(G\) has no isolated
  vertices, it follows that \(\supp(u) = V\).  Now the off-diagonal
  entries of the equality constraint of~\eqref{eq:27} show that \(G =
  K_V\) and that \(w_{ij} = 4 u_i u_j\) for each \(ij \in E =
  \tbinom{V}{2}\).
\end{proof}

Instances of MaxCut of the form described by
\Cref{prop:maxcut-rankone} are still NP-hard.  Indeed, it is easy to
see that they may be reformulated as
\(\max\setst{(\iprodt{\incidvector{U}}{u})(\iprodt{\incidvector{V
      \setminus U}}{u})}{\emptyset \neq U \subsetneq V}\).  The latter
problem is easily seen to include the partition problem.

\subsection{Dual Integrality for the MaxCut SDP}
\label{sec:maxcut-dual-int}

As described in \Cref{sec:drawbacks-rank-one}, our theory does not
apply directly to the embedding used in the MaxCut
SDP~\eqref{eq:maxcut-sdp}.  To formulate~\eqref{eq:maxcut-sdp} in our
format, first rewrite it as
\begin{equation}
  \label{eq:maxcut-sdp-aux}
  \begin{array}[!h]{rll}
    \text{Maximize}   & \iprod{0 \oplus\tfrac{1}{4}\Laplacian{G}(w)}{\Yhat} &                          \\[4pt]
    \text{subject to} & \iprod{\oprodsym{e_i}}{\Yhat} = 1                   & \forall i \in \zlift{V}, \\[4pt]
                      & \Yhat \in \Psdhat{\Vspc},                           &                          \\
  \end{array}
\end{equation}
and then perform the change of variable \(\Yhat \mapsto
\Bhat\Yhat\Bhat^{\transp} = \Xhat\), where
\begin{equation*}
  \Bhat \coloneqq
  \frac{1}{2}
  \begin{bmatrix}
    2     & 0^{\transp} \\
    \ones & I           \\
  \end{bmatrix}
\end{equation*}
to get the equivalent SDP
\begin{equation}
  \label{eq:maxcut-sdp-aux2}
  \begin{array}[!h]{rll}
    \text{Maximize}   & \iprod{0 \oplus \Laplacian{G}(w)}{\Xhat}                &                  \\[4pt]
    \text{subject to} & \iprod{\oprodsym{e_0}}{\Xhat} = 1,                      &                  \\[4pt]
                      & \iprod{2\Symmetrize(\oprod{e_i}{(e_i-e_0)})}{\Xhat} = 0 & \forall i \in V, \\[4pt]
                      & \Xhat \in \Psdhat{\Vspc}.                               &                  \\
  \end{array}
\end{equation}
Finally, add the redundant constraints \(\diag(\Xhat[V]) \geq 0\) to
get the \emph{homogeneous MaxCut SDP}:
\begin{equation}
  \label{eq:maxcut-sdp-homog}
  \begin{array}[!h]{rll}
    \text{Maximize}        & \iprod{0 \oplus \Laplacian{G}(w)}{\Xhat}      & \\[4pt]
    \text{subject to}\quad & \text{\(\Xhat\) satisfies~\eqref{eq:SDP-01}}, & \\[4pt]
                           & \Xhat \in \Psdhat{\Vspc}.                     & \\
  \end{array}
\end{equation}
Note that the change of variable is a linear automorphism
of~\(\Symhat{\Vspc}\) that preserves rank, so we are not giving ourselves
any undue advantage by choosing this embedding.

The dual SDP of~\eqref{eq:maxcut-sdp-homog} is
\begin{equation}
  \label{eq:maxcut-sdp-dual}
  \begin{array}[!h]{rll}
    \text{Minimize}   & \eta
                      & \\[4pt]
    \text{subject to} & \begin{bmatrix}
                          \eta & -u^{\transp} \\
                          -u   & \Diag(2u-z)
                        \end{bmatrix}
                        - \Shat
                        =
                        \begin{bmatrix}
                          0 & 0^{\transp} \\
                          0 & \Laplacian{G}(w)
                        \end{bmatrix},
                      & \\[10pt]
                      & \Shat \in \Psdhat{\Vspc},
                        \eta \in \Reals,\,
                        u \in \Reals^V,\,
                        z \in \Reals_+^V.
  \end{array}
\end{equation}
Upon adding the integrality constraint to~\eqref{eq:maxcut-sdp-dual}
(and assuming integrality of \(w \in \Integers_+^E\)),
we obtain
\begin{equation}
  \label{eq:maxcut-sdp-integer-dual-pre}
  \begin{array}[!h]{rll}
    \text{Minimize}  & \iprodt{\ones}{m}
                     & \\[4pt]
    \text{subject to} & m \colon \Powerset{V} \to \Integers_+,
                     & \\[4pt]
                     & u = \sum_{A \subseteq V} m_A \incidvector{A},
                     & \\[4pt]
                     & \iprodt{\incidvector{\binom{V}{i\in}}}{m}
                     \leq
                       2u_i - \iprodt{\incidvector{\delta(i)}}{w}
                     & \forall i \in V, \\[8pt]
                     & \iprodt{\incidvector{\binom{V}{ij\subseteq}}}{m}
                       =
                       [ij \in E] w_{ij}
                     & \forall ij \in \binom{V}{2},
  \end{array}
\end{equation}
which may be finally simplified to
\begin{subequations}
  \label{eq:maxcut-integer-dual}
  \begin{alignat}{3}
    \label{eq:maxcut-integer-dual-obj}
    \text{Minimize} \quad & \iprodt{\ones}{m}
    & \\[0pt]
    \label{eq:maxcut-integer-dual-domain}
    \text{subject to}\quad & m \colon \Powerset{V} \setminus \set{\emptyset} \to \Integers_+,
    & \\[0pt]
    \label{eq:maxcut-integer-dual-clique}
    & \supp(m) \subseteq \Kcal(G), & \\[0pt]
    \label{eq:maxcut-integer-dual-vertex}
    & \iprodt{\incidvector{\delta(i)}}{w}
      \leq
      \iprodt{\incidvector{\binom{V}{i\in}}}{m}
    & \qquad & \forall i \in V, \\[0pt]
    \label{eq:maxcut-integer-dual-edge}
    & \iprodt{\incidvector{\binom{V}{ij\subseteq}}}{m} = w_{ij}
    & \qquad & \forall ij \in E.
  \end{alignat}
\end{subequations}

The next result yields a closed formula for the unique optimal solution
of~\eqref{eq:maxcut-integer-dual}:
\begin{theorem}
  \label{thm:maxcut-integer-dual}
  Let \(G = (V,E)\) be a graph and let \(w \colon E \to \Integers_+\).
  Then the optimization problem
  \begin{subequations}
    \label{eq:maxcut-integer-dual-leq}
    \begin{alignat}{3}
      \label{eq:maxcut-integer-dual-leq-obj}
      \text{Minimize} \quad & \iprodt{\ones}{m}
      & \\[0pt]
      \label{eq:maxcut-integer-dual-leq-domain}
      \text{subject to}\quad & m \colon \Powerset{V} \setminus \set{\emptyset} \to \Integers_+,
      & \\[0pt]
      \label{eq:maxcut-integer-dual-leq-clique}
      & \supp(m) \subseteq \Kcal(G), & \\[0pt]
      \label{eq:maxcut-integer-dual-leq-vertex}
      & \iprodt{\incidvector{\delta(i)}}{w}
        \leq
        \iprodt{\incidvector{\binom{V}{i\in}}}{m}
      & \qquad & \forall i \in V, \\[0pt]
      \label{eq:maxcut-integer-dual-leq-edge}
      & \iprodt{\incidvector{\binom{V}{ij\subseteq}}}{m} \leq w_{ij}
      & \qquad & \forall ij \in E.
    \end{alignat}
  \end{subequations}
  has a unique optimal solution~\(m^*\), and it satisfies
  \(\supp(m^*) \subseteq E\) and
  \(m^*\mathord{\restriction}_E = w\).
\end{theorem}

\begin{proof}
  Let \(m_w \colon \Powerset{V} \to \Integers_+\) be the zero
  extension of~\(w\), that~is, \(\supp(m_w) \subseteq E\) and
  \(m_w\mathord{\restriction}_E = w\).  It is easy to check that
  \(m_w\) is feasible in~\eqref{eq:maxcut-integer-dual-leq}.  Let
  \(m^* \colon \Powerset{V} \to \Integers_+\) be an optimal solution
  for~\eqref{eq:maxcut-integer-dual-leq}; one exists since there exist
  feasible solutions and the objective value of every feasible
  solution is a nonnegative integer.  We will prove that
  \begin{equation}
    \label{eq:17}
    m^* = m_w.
  \end{equation}

  The key part of the proof is to show that
  \begin{equation}
    \label{eq:18}
    \supp(m^*) \subseteq \tbinom{V}{1} \cup \tbinom{V}{2}.
  \end{equation}
  Let \(C \in \supp(m^*)\).  We claim that
  \begin{equation}
    \label{eq:19}
    \text{\(\tilde{m} \coloneqq
      m^* - e_C - \big(\card{C}-2\big)\incidvector{\binom{C}{1}} +
      \incidvector{E[C]}\) is feasible for~\eqref{eq:maxcut-integer-dual-leq}}.
  \end{equation}
  For every \(i \in V\), we have
  \begin{equation*}
    \incidvector{\binom{V}{i\in}}^{\transp}
    \Big({
      e_C + \big(\card{C}-2\big)\incidvector{\binom{C}{1}}
    }\Big)
    =
    [i \in C] + \big(\card{C}-2\big)[i \in C]
    =
    [i \in C]\big(\card{C}-1\big)
    =
    \incidvector{\binom{V}{i\in}}^{\transp}
    \incidvector{E[C]},
  \end{equation*}
  so \eqref{eq:maxcut-integer-dual-leq-vertex} holds
  for~\(\tilde{m}\).  For every \(ij \in E\) we have
  \begin{equation*}
    \incidvector{\binom{V}{ij\subseteq}}^{\transp}
    \Big({
      e_C + \big(\card{C}-2\big)\incidvector{\binom{C}{1}}
    }\Big)
    =
    \big[ij \in E[C]\big]
    =
    \incidvector{\binom{V}{ij\subseteq}}^{\transp}
    \incidvector{E[C]},
  \end{equation*}
  so \eqref{eq:maxcut-integer-dual-leq-edge} holds for~\(\tilde{m}\).
  If \(C\) is a singleton, then \(\tilde{m} = m^*\) and
  \eqref{eq:maxcut-integer-dual-leq-domain} holds.  So,
  in~verifying~\eqref{eq:maxcut-integer-dual-leq-domain} for
  \(\tilde{m}\), we may assume \(\card{C}\geq 2\).  We will prove
  that~\eqref{eq:maxcut-integer-dual-leq-domain} holds
  for~\(\tilde{m}\) by showing that
  \begin{equation}
    \label{eq:20}
    \mbar \coloneqq m^* - e_C \geq \big(\card{C}-2\big) \incidvector{\binom{C}{1}};
  \end{equation}
  then \eqref{eq:maxcut-integer-dual-leq-clique} for~\(\tilde{m}\)
  will also follow, thus completing the proof of~\eqref{eq:19}.

  Note that \(\mbar \geq 0\).  Let \(i \in V\).  Then
  \begin{alignat*}{2}
    \incidvector{\delta(i)}^{\transp} w
    & \leq
    \incidvector{\binom{V}{i\in}}^{\transp} m^*
    & \quad & \text{by~\eqref{eq:maxcut-integer-dual-leq-vertex}}
    \\
    & =
    \incidvector{\binom{V}{i\in}}^{\transp} \mbar
    + [i \in C]
    & \quad & \text{since \(m^* = \mbar + e_C\)}
    \\
    & \leq
    \mbar_{\set{i}}
    + \sum_{\mathclap{j \in V \setminus \set{i}}} \incidvector{\binom{V}{ij\subseteq}}^{\transp} \mbar
    + [i \in C]
    & \quad & \text{since \(\incidvector{\binom{V}{i\in}}
      \leq e_{\set{i}} + \sum_{\mathclap{j \in V \setminus \set{i}}}\incidvector{\binom{V}{ij \subseteq}}\)}
    \\
    & =
    \mbar_{\set{i}}
    + \sum_{\mathclap{j \in N(i)}} \incidvector{\binom{V}{ij\subseteq}}^{\transp} \mbar
    + [i \in C]
    & \quad & \text{by~\eqref{eq:maxcut-integer-dual-leq-clique}}
    \\
    & =
    \mbar_{\set{i}}
    + \sum_{\mathclap{j \in N(i)}} \incidvector{\binom{V}{ij\subseteq}}^{\transp} m^*
    - \sum_{\mathclap{j \in N(i)}} \incidvector{\binom{V}{ij\subseteq}}^{\transp} e_C
    + [i \in C]
    & \quad & \text{since \(\mbar = m^* - e_C\)}
    \\
    & \leq
    \mbar_{\set{i}}
    + \sum_{j \in N(i)} w_{ij}
    - \card{\delta(i) \cap E[C]}
    + [i \in C]
    & \quad & \text{by~\eqref{eq:maxcut-integer-dual-leq-edge}}
    \\
    & =
    \mbar_{\set{i}}
    + \iprodt{\incidvector{\delta(i)}}{w} - [i\in C]\big(\card{C}-2\big)
    & \quad & \text{since \(\card{\delta(i) \cap E[C]} = [i\in C]\big(\card{C}-1\big)\)}.
  \end{alignat*}
  This proves~\eqref{eq:20}, and thus completes the proof
  of~\eqref{eq:19}.

  We have
  \begin{equation*}
    \iprodt{\ones}{m^*}
    -
    \iprodt{\ones}{\tilde{m}}
    =
    \iprodt{\ones}{
      \Big({
        e_C + \big(\card{C}-2\big)\incidvector{\binom{C}{1}}
      }\Big)
    }
    - \iprodt{\ones}{\incidvector{E[C]}}
    =
    1 + \card{C}\big(\card{C}-2\big) - \binom{\card{C}}{2}
    =
    \frac{1}{2}\big(\card{C}-1\big)\big(\card{C}-2\big).
  \end{equation*}
  Optimality of~\(m^*\) and \eqref{eq:19} imply that \(\card{C} \in
  \set{1,2}\).  This concludes the proof of~\eqref{eq:18}.

  By summing the vertex
  constraints~\eqref{eq:maxcut-integer-dual-leq-vertex} and
  using~\eqref{eq:18}, we obtain
  \begin{equation}
    \label{eq:21}
    2 \ones^{\transp} w \leq \bigg(\sum_{A \subseteq V} \card{A}e_A\bigg)^{\transp} m^*
    = \incidvector{\binom{V}{1}}^{\transp} m^* + 2 \incidvector{\binom{V}{2}}^{\transp} m^*.
  \end{equation}
  By summing the edge
  constraints~\eqref{eq:maxcut-integer-dual-leq-edge} and
  using~\eqref{eq:18}, we obtain
  \begin{equation}
    \label{eq:22}
    \incidvector{\binom{V}{2}}^{\transp} m^*
    =
    \bigg(\sum_{A \subseteq V} \binom{\card{A}}{2} e_A\bigg)^{\transp} m^*
    \leq
    \ones^{\transp} w.
  \end{equation}
  It follows from~\eqref{eq:18}, \eqref{eq:21}, and~\eqref{eq:22} that
  \begin{equation}
    \label{eq:23}
    \iprodt{\ones}{m_w}
    =
    \iprodt{\ones}{w}
    \leq
    \iprodt{\incidvector{\binom{V}{1}}}{m^*} + \iprodt{\incidvector{\binom{V}{2}}}{m^*}
    =
    \iprodt{\ones}{m^*}.
  \end{equation}
  Equality throughout in~\eqref{eq:23} implies that each constraint
  in~\eqref{eq:maxcut-integer-dual-leq-vertex}
  and~\eqref{eq:maxcut-integer-dual-leq-edge} holds with equality
  for~\(m^*\), so that \(m^*\) is feasible
  for~\eqref{eq:maxcut-integer-dual}.  The latter fact, together
  with~\eqref{eq:18}, easily implies that \(m^* = m_w\).
\end{proof}

Note that \Cref{thm:maxcut-integer-dual} does not characterize total
dual integrality of the MaxCut SDP~\eqref{eq:maxcut-sdp} since it only
identifies integral dual optimal solutions when the weight function
\(w\) on the edges is nonnegative.  We postpone the discussion of dual
integrality for not necessarily nonnegative weight functions
to~\Cref{sec:maxcut-nonneg}.

\section{Conclusion and Future Directions}
\label{sec:conclusion}

We have introduced a primal-dual symmetric notion integrality in SDPs
in~\Cref{def:primal-int,def:dual-int}; see also
conditions~\eqref{eq:primal-int} and~\eqref{eq:dual-int}.  This
enabled the statement in \Cref{thm:1} of the SDP version of the
LP-based \Cref{thm:LP-chain}.  Then, by relying on our generalization
of~\Cref{cor:Hoffman} in~\Cref{cor:3.1}, and the notion of \emph{total
  dual integrality} through a linear map in~\Cref{def:TDI}, we
described sufficient conditions for exactness of the (primal) SDP
formulation in \Cref{thm:tdi-projection} and equality throughout the
chain from \Cref{thm:1} in \Cref{thm:eq-throughout}.  We also
characterized the semidefinite notions of TDIness in the LP case
(\Cref{thm:SDP-TDI-LP}) and all variants of the theta function
(\Cref{thm:theta-TDI}) via natural conditions.  Finally, in
\Cref{thm:maxcut-integer-dual}, we completely determined the optimal
solutions for the integer dual SDP for the MaxCut SDP when the weight
function on the edges of the graph is nonnegative.

Our approach leads to several other interesting research directions.
We start with:
\begin{problem}
  Obtain a primal-dual symmetric integrality condition for SDPs that
  applies to arbitrary ILPs, not just binary ones.
\end{problem}

The theory of total dual integrality for LPs is considered well
understood.  Our work raises new issues, related to the interplay
between total dual integrality and extended formulations in LP; the
latter area has received a lot of attention recently.  More
concretely, one may define a system of linear inequalities \(Ax \leq
b\) on~\(\Reals^n\) to be TDI through a linear map \(L \colon \Reals^k
\to \Reals^n\) if, for every integral \(c \in \Integers^k\), the LP
dual to \(\sup\setst{\iprod{L(c)}{x}}{Ax \leq b}\) has an integral
optimal solution if its optimal value is finite.
\begin{problem}
  Are there compact formulations for classical combinatorial
  optimization problems (e.g., maximum weight \(r\)-arborescences,
  minimum spanning trees) that are TDI through the corresponding
  lifting maps?  Do these lead to new min-max theorems?
\end{problem}
\begin{problem}
  Let \(Ax \leq b\) be a system of linear inequalities on~\(\Reals^n\)
  and \(L \colon \Reals^k \to \Reals^n\) a linear map such that for
  \(P \coloneqq \setst{x \in \Reals^n}{Ax \leq b}\) the projection
  \(L^*(P)\) is integral.  Does there exist a TDI system \(Cx \leq d\)
  in~\(\Reals^n\) with \(d\) integral such that \(L^*(P) =
  L^*(\setst{x \in \Reals^n}{Cx \leq d})\)?
\end{problem}

The next problem is somewhat more open ended:
\begin{problem}
  What is the relation between total dual integrality and the integer
  decomposition property (see~\cite[sec~22.10]{Schrijver86a}), of
  which our dual integrality condition in~\Cref{def:dual-int} is
  reminiscent?
\end{problem}

In \Cref{sec:maxcut} we studied dual integrality of MaxCut SDP with
nonnegative weight functions, and we discuss
in~\Cref{sec:maxcut-nonneg} the issues that arise when we allow
weights of arbitrary signs.  These issues suggest further research
directions.  One may define a refinement of the notion of total dual
integrality restricted to a rational convex cone \(\mathbb{K}
\subseteq \Reals^k\); there, one would only require the dual SDP to
have an integral solution for primal objective functions of the form
\(\Xhat \mapsto \iprod{\Lcal(c)}{\Xhat}\) with \(c \in \mathbb{K}\).
In this context, it seems misleading to use the term \emph{total} dual
integrality; perhaps \(\mathbb{K}\)-dual integrality would seem more
adequate.
\begin{problem}
  \label{prob:KDI}
  Adapt \Cref{thm:tdi-projection} to the notion of \(\mathbb{K}\)-dual
  integrality; how should the set~\(\Cscr\) be modified
  using~\(\mathbb{K}\) to yield an integral convex set?
\end{problem}

Concerning the semidefinite notion of TDIness, one may ask for a
characterization of total dual integrality of other SDP formulations,
such as the application of lift-and-project hierarchies
(see~\cite{Laurent03a}) to ILP formulations of combinatorial
optimization problems.  One possible instance is the following:
\begin{problem}
  Given \(k \geq 1\) and the \(LS_+\) operator of Lovász and
  Schrijver~\cite{LovaszS91a} (called~\(N_+\) in their paper),
  determine the class of graphs for which the \(k\)th iterate of the
  \(LS_+\) operator applied to the system
  \begin{equation}
    \label{eq:16}
    x \geq 0,\qquad x_i+x_j \leq 1 \quad\forall ij \in E
  \end{equation}
  yields a TDI system through the appropriate lifting, leading to a
  minmax relation involving stable sets in such graphs.
\end{problem}

Still in the realm of SDPs, one may ask for notions of
\emph{exactness} other than integrality, as well as their dual
counterparts.  For~instance, many problems in continuous mathematics,
such as control theory, lead to nonconvex optimization problems where
the variable matrix is required to be rank-one or of restricted rank.
However, the entries of such a matrix may define a continuous curve
rather than taking on only finitely many values.  For a general convex
relaxation framework working with such formulations,
see~\cite{KojimaT00a}.
\begin{problem}
  Obtain systematic, primal-dual symmetric conditions for exactness in
  SDP relaxations for continuous problems.
\end{problem}

Finally, one may consider the problem of defining integrality in a
systematic and primal-dual symmetric way for convex optimization
problems in other forms.  This is especially challenging since a dual
integrality notion, even~in the polyhedral case, is inherently
dependent on the algebraic representation of the problem, not only on
its geometry.

\appendix

\section{Rank Constraint in Dual SDP of Trace Formulation for Theta}
\label{sec:rank-theta-tr}

In \Cref{sec:drawbacks-rank-one} we showed that the rank-one
constraint for the dual SDP of a formulation of the theta function is
not very interesting.  There, the formulation we used was based on our
chosen embedding into the lifted space \(\Symhat{\Vspc}\), which
requires the constraints~\eqref{eq:SDP-01}.  One might argue that the
rank-one constraint might make more sense for the dual SDP of the
following, probably more popular, formulation of \(\theta(G,w)\) for a
graph \(G = (V,E)\) and \(w \colon V \to \Reals_+\):
\begin{subequations}
  \label{eq:theta3}
  \begin{alignat}{3}
    \text{Maximize} \quad & \iprod{\oprodsym{\sqrt{w}}}{X} & \\[2pt]
    \text{subject to}\quad & \iprod{I}{X} = 1,
    & \\[2pt]
    \label{eq:theta3-edge}
    & \iprod{\Symmetrize(\oprod{e_i}{e_j})}{X} = 0
    & \quad & \forall ij \in E, \\[2pt]
    & X \in \Psd{V}. &
  \end{alignat}
\end{subequations}
We will show that the rank-one constraint is not very meaningful even
in the dual of the following SDP formulation of~\(\theta'(G,w)\):
\begin{equation}
  \label{eq:theta3'}
  \begin{array}[!h]{rll}
    \text{Maximize} \quad & \iprod{\oprodsym{\sqrt{w}}}{X} & \\[2pt]
    \text{subject to}\quad & \iprod{I}{X} = 1,
    & \\[2pt]
    & \iprod{\Symmetrize(\oprod{e_i}{e_j})}{X} = 0
    & \forall ij \in E, \\[2pt]
    & \iprod{\Symmetrize(\oprod{e_i}{e_j})}{X} \geq 0
    & \forall ij \in \overline{E}, \\[2pt]
    & X \in \Psd{V}, &
  \end{array}
\end{equation}
where \(\overline{E} \coloneqq \tbinom{V}{2} \setminus E\) is the edge
set of~\(\overline{G}\).  Note that the dual SDP of~\eqref{eq:theta3'} is
\begin{equation*}
  \begin{array}[!h]{rll}
    \text{Minimize} \quad & \lambda & \\[0pt]
    \text{subject to}\quad & \lambda I + {\displaystyle \sum_{ij \in \tbinom{V}{2}}} y_{ij}
    \Symmetrize(\oprod{e_i}{e_j}) \succeq \oprodsym{\sqrt{w}}, & \\
    & y\mathord{\restriction}_{\overline{E}} \leq 0, & \\
  \end{array}
\end{equation*}
or, equivalently,
\begin{equation}
  \label{eq:theta2'}
  \begin{array}[!h]{rll}
    \text{Minimize} \quad & \lambda & \\[2pt]
    \text{subject to}\quad & \lambda I + A - \overline{A} - S = \oprodsym{\sqrt{w}}, & \\
    & S \succeq 0, & \\
    & A \in \Ascr_G, & \\
    & \overline{A} \in \Ascr_{\overline{G}} \cap \Symnonneg{V}. & \\
  \end{array}
\end{equation}
Note that the dual of the formulation~\eqref{eq:theta3} for
\(\theta(G,w)\) is obtained from~\eqref{eq:theta2'} by dropping the
variable matrix~\(\overline{A}\), i.e., by~setting \(\overline{A} =
0\).  Hence, every feasible solution for the dual SDP
of~\eqref{eq:theta3} is feasible in~\eqref{eq:theta2'}.

One could formulate \(\theta^+(G,w)\) similarly as~\eqref{eq:theta3},
by replacing the equality in the edge
constraints~\eqref{eq:theta3-edge} with~`\(\leq\)'.  The corresponding
dual SDP is obtained from~\eqref{eq:theta2'} by setting \(\overline{A}
= 0\) and requiring \(A \in \Ascr_G \cap \Symnonneg{V}\).  Again, the
feasible region of this dual SDP is a subset of the feasible region
of~\eqref{eq:theta2'}.

The embedding of stable sets in~\(G\) as feasible solutions
of~\eqref{eq:theta3'} goes as follows: if \(U \subseteq V\) is a
stable set in~\(G\) with positive weight
\(\iprodt{w}{\incidvector{U}}\), then \(X \coloneqq
\paren{\iprodt{w}{\incidvector{U}}}^{-1} \oprodsym{(\sqrt{w} \hprod
  \incidvector{U})}\) is feasible in~\eqref{eq:theta3'}, with
objective value \(\iprodt{w}{\incidvector{U}}\).  The normalization
factor and the square root in the definition of~\(X\) already hint
that this formulation does not play so well with integrality.

\begin{proposition}
  Let \(G = (V,E)\) be a graph and let \(w \in \Reals_{++}^V\).  If
  there exists a feasible solution \((\lambda,A,\overline{A},S)\)
  for~\eqref{eq:theta2'} such that \(\rank(S) \leq 1\), then
  \(\overline{G}\) is bipartite.
\end{proposition}

\begin{proof}
  Suppose \(S = \oprodsym{s}\) for some \(s \in \Reals^V\).  Then
  \begin{equation}
    \label{eq:24}
    \lambda I + A = \oprodsym{s} + \oprodsym{\sqrt{w}} + \overline{A}.
  \end{equation}
  Apply \(\diag\) to both sides of~\eqref{eq:24} to get \(\lambda
  \ones = (s \hprod s) + w\).  Hence, \(\lambda \ones \geq w\) and
  there exists \(U \subseteq V\) such that
  \begin{equation}
    \label{eq:25}
    s = \Diag(2\incidvector{U}-\ones)\sqrt{\lambda\ones-w}.
  \end{equation}
  Let \(ij \in \overline{E}\).  Specialize~\eqref{eq:24} to the
  \(ij\)th entry to get
  \begin{equation}
    \label{eq:26}
    0 = s_i s_j + \sqrt{w_i w_j} + \overline{A}_{ij}
    \geq (-1)^{[i\not\in U]+[j\not\in U]}\sqrt{(\lambda-w_i)(\lambda-w_j)}
    +\sqrt{w_i w_j}.
  \end{equation}
  If \(i,j \in U\) or \(i,j \in \overline{U} \coloneqq V \setminus
  U\), then the RHS of~\eqref{eq:26} is positive, since \(w \in
  \Reals_{++}^V\).  This contradiction shows that \(G[U] = K_U\) and
  \(G[\overline{U}] = K_{\overline{U}}\), so \(\overline{G}\) is
  bipartite with color classes~\(U\) and~\(\overline{U}\).
\end{proof}

By our previous discussion, the dual SDPs of the above formulations
of~\(\theta\), \(\theta'\), and~\(\theta^+\) only have rank-one slacks
when \(\overline{G}\) is bipartite (whence \(G\) is perfect).

We point out, however, that another low-rank constraint for the dual
SDP for~\(\theta'\) does in fact yield a useful and almost exact
formulation for the chromatic number of a graph \(G = (V,E)\), via the
circular chromatic number.  We first describe the vector chromatic
number, introduced in~\cite{KargerMS98a}.  Suppose \(G\) has at least
one edge.  The \emph{vector chromatic number}~\(\chi_v(G)\) of~\(G\)
is the optimal value of the following optimization problem
\begin{equation}
  \label{eq:vector-chromatic}
  \begin{array}[!h]{rll} \text{Minimize} \quad & \tau & \\[2pt]
    \text{subject to}\quad & \diag(Y) = \ones, & \\[2pt]
    & \iprod{\Symmetrize(\oprod{e_i}{e_j})}{Y} \leq -\tfrac{1}{\tau-1}
    & \forall ij \in E,\\[2pt]
    & Y \in \Psd{V}, & \\[2pt]
    & \tau \geq 2. & \\
  \end{array}
\end{equation}
It is not hard to see that the map \((S,\lambda) \mapsto
\tfrac{1}{\lambda-1}S\) maps bijectively the feasible region
of~\eqref{eq:theta2'} applied to~\(\overline{G}\) to the feasible
region of~\eqref{eq:vector-chromatic} and preserves the objective
value.  Hence, \(\chi_v(G) = \theta'(\overline{G})\).  This suggests
the following alternative SDP formulation for~\(\chi_v(G)\):
\begin{equation}
  \label{eq:vector-chromatic2}
  \begin{array}[!h]{rll} \text{Minimize} \quad & \sigma & \\[2pt]
    \text{subject to}\quad & \diag(Y) = \ones, & \\[2pt]
    & \iprod{\Symmetrize(\oprod{e_i}{e_j})}{Y} \leq \sigma
    & \forall ij \in E,\\[2pt]
    & Y \in \Psd{V}, & \\[2pt]
    & \tau \geq 2. & \\
  \end{array}
\end{equation}
Any optimal solution \(\sigma^*\) lies in \(\halfopen{{-1}}{0}\) and
leads to the optimal value \(\tau^* \coloneqq 1-1/\sigma^*\)
for~\eqref{eq:vector-chromatic}.

Consider next the \emph{circular chromatic number}~\(\chi_c(G)\)
of~\(G\), which can be defined as the optimal value of the
optimization problem
\begin{equation}
  \label{eq:circular-chromatic}
  \begin{array}[!h]{rll} \text{Minimize} \quad & \tau & \\[2pt]
    \text{subject to}\quad & y \colon V \to S^1, & \\[2pt]
    & \phi_{ij} \geq \frac{2\pi}{\tau} & \forall ij \in E,\\[2pt]
    & \tau \geq 2, & \\
  \end{array}
\end{equation}
where \(S^1\) denotes the unit sphere in~\(\Reals^2\) and \(\phi_{ij}
\in [0,\pi]\) is the angle between \(y_i\) and~\(y_j\).  This
formulation can be seen in~\cite{DeVos}; see~\cite{Zhu01a} for further
properties of~\(\chi_c\).  Since \(\cos\) is monotone decreasing
on~\([0,\pi]\), we can rewrite the latter optimization problem using
Gram matrices as
\begin{equation}
  \label{eq:circular-chromatic-Gram}
  \begin{array}[!h]{rll} \text{Minimize} \quad & \tau & \\[2pt]
    \text{subject to}\quad & \diag(Y) = \ones, & \\[2pt]
    & \iprod{\Symmetrize(\oprod{e_i}{e_j})}{Y} \leq \cos\frac{2\pi}{\tau}
    & \forall ij \in E,\\[2pt]
    & Y \in \Psd{V}, & \\[2pt]
    & \rank(Y) = 2, & \\[2pt]
    & \tau \geq 2. & \\
  \end{array}
\end{equation}
Finally, since \(f \colon \tau \in \halfopen{2}{\infty} \mapsto
\cos\frac{2\pi}{\tau} \in \halfopen{{-1}}{1}\) is a monotone
increasing bijection, we see that, if \(\sigma^*\) is the optimal
value of~\eqref{eq:vector-chromatic2} with the extra constraint
\(\rank(Y) = 2\), then \(\chi_c(G) = f^{-1}(\sigma^*)\).  One can then
read off the chromatic number of~\(G\) since \(\chi(G) =
\ceiling{\chi_c(G)}\).

Note, however, that this dual formulation required quite a lot of
\emph{ad~hoc} treatment.

\section{The MaxCut Problem and Nonnegative Weights}
\label{sec:maxcut-nonneg}

One may wonder whether \Cref{thm:maxcut-integer-dual} may be extended
to arbitrary weight functions \(w \colon E \to \Integers\), not~just
nonnegative weights.  Such an extension might be used to characterize
the graphs~\(G\) for which the system defining the MaxCut
SDP~\eqref{eq:maxcut-sdp} is TDI through \(w \in \Reals^E \mapsto 0
\oplus \Laplacian{G}(w) \eqqcolon \Lcal(w)\); by
\Cref{thm:maxcut-integer-dual} such graphs forms a subset of the
bipartite graphs.  Then we would be able to obtain the \emph{cut
  polytope} \(\conv\setst{\incidvector{\delta(S)}}{\emptyset \neq S
  \subsetneq V}\) of any such graph~\(G\) as a projection of the
feasible region of~\eqref{eq:maxcut-sdp} via~\(\Lcal^*\).  However,
due to constraints~\eqref{eq:maxcut-integer-dual-leq-edge}, if \(w
\colon E \to \Integers\) has a negative entry, problem
\eqref{eq:maxcut-integer-dual-leq} is infeasible.  One~may attempt to
``fix'' this issue by adding to~\eqref{eq:maxcut-sdp-homog} the
redundant constraint \(\Lcal^*(\Xhat) = \Laplacian{G}^*(\Xhat[V]) \geq
0\).  Note that this is similar to the redundant
constraint~\eqref{eq:SDP-01-nonneg} added in our chosen embedding,
which is fundamental for dealing with \(w \in \Reals^V \setminus
\Reals_+^V\) for the~\(\theta\) function; in both cases, the redundant
constraint comes from the projection~\(\Lcal^*\).  The dual SDP is
then obtained from~\eqref{eq:maxcut-sdp-dual} by replacing the
occurrence of~\(\Laplacian{G}(w)\) in the RHS with
\(\Laplacian{G}(w+y)\), where \(y \in \Reals_+^E\) is a new variable.
Optimal solutions for the corresponding integer dual SDP are described
by the next result:
\begin{corollary}
  \label{cor:maxcut-integer-dual2}
  Let \(G = (V,E)\) be a graph and let \(w \colon E \to \Integers\).
  Then the optimization problem
  \begin{subequations}
    \label{eq:maxcut-integer-dual2-leq}
    \begin{alignat}{3}
      \label{eq:maxcut-integer-dual2-leq-obj}
      \text{Minimize} \quad & \iprodt{\ones}{m}
      & \\[0pt]
      \label{eq:maxcut-integer-dual2-leq-domain}
      \text{subject to}\quad & m \colon \Powerset{V} \setminus \set{\emptyset} \to \Integers_+,
      & \\[0pt]
      \label{eq:maxcut-integer-dual2-leq-clique}
      & \supp(m) \subseteq \Kcal(G), & \\[0pt]
      \label{eq:maxcut-integer-dual2-y}
      & y \in \Reals_+^E, & \\[0pt]
      \label{eq:maxcut-integer-dual2-leq-vertex}
      & \iprodt{\incidvector{\delta(i)}}{(w+y)}
        \leq
        \iprodt{\incidvector{\binom{V}{i\in}}}{m}
      & \qquad & \forall i \in V, \\[0pt]
      \label{eq:maxcut-integer-dual2-leq-edge}
      & \iprodt{\incidvector{\binom{V}{ij\subseteq}}}{m} \leq w_{ij} + y_{ij}
      & \qquad & \forall ij \in E.
    \end{alignat}
  \end{subequations}
  has a unique optimal solution \((m^*,y^*)\), and it satisfies
  \(\supp(m^*) \subseteq E\), and for each \(e \in E\),
  \begin{gather*}
    m_e^* = [w_e \geq 0]w_e,
    \qquad
    y_e^* = -[w_e < 0]w_e.
  \end{gather*}
\end{corollary}
\begin{proof}
  Let \((\mbar,\ybar)\) be feasible.  By
  \eqref{eq:maxcut-integer-dual2-leq-edge}, we have \(w+\ybar \geq 0\)
  so \(y \geq y^*\).  By \Cref{thm:maxcut-integer-dual}, the
  optimization problem \eqref{eq:maxcut-integer-dual2-leq} with the
  extra constraint \(y = \ybar\) has a unique optimal solution, and
  its optimal value is \(\iprodt{\ones}{(w+\ybar)}\), which is greater
  than or equal to \(\iprodt{\ones}{(w+y^*)}\), the objective value of
  the feasible solution \((m^*,y^*)\).
\end{proof}

Even though \Cref{cor:maxcut-integer-dual2} shows how the dual SDP for
MaxCut with an extra (redundant) constraint may have integral
solutions, the optimal value is always nonnegative.  The deeper
problem here is that the MaxCut SDP~\eqref{eq:maxcut-sdp} is not tight
for arbitrary weights~\(w\), even if the underlying graph is
bipartite.  Hence, if~\(\Cscr \subseteq \Reals^E\) is the projection
of the feasible region of~\eqref{eq:maxcut-sdp} via~\(\Lcal^*\), we
cannot even expect \(\Cscr = \Cscr_I\), let alone total dual
integrality of the defining system.

To see this, first note that, for a graph \(G = (V,E)\) and weights
\(w \colon E \to \Reals\), we should redefine the maximum cut problem
as the optimization problem
\(\sup\setst{\iprodt{w}{\incidvector{\delta(U)}}}{\emptyset \neq U
  \subsetneq V}\); when \(w \geq 0\), since \(\delta(\emptyset) =
\delta(V) = \emptyset\), it was harmless to keep both trivial sets \(U
= \emptyset\) and \(U = V\) in the feasible set.  Correspondingly, in
the MaxCut SDP~\eqref{eq:maxcut-sdp}, the feasible solution \(X
\coloneqq \oprodsym{\ones}\) shows that the optimal value is always
nonnegative, \emph{even when \(w\) is negative and \(G\) is
  connected\/}!  To prevent these trivial solutions from being
feasible in a modified MaxCut SDP, one may add the constraint
\(\iprod{\oprodsym{\ones}}{X} \leq (\card{V}-2)^2\); to see where the
RHS comes from, note that
\[
\max\setst{\iprod{\oprodsym{\ones}}{\oprodsym{s_U}}}{\emptyset \neq U
  \subsetneq V}
=
(\card{V}-2)^2,
\]
where \(s_U \coloneqq 2\incidvector{U}-\ones\) for each \(U \subseteq
V\).  These considerations lead us to strengthen~\eqref{eq:maxcut-sdp} as
\begin{equation}
  \label{eq:maxcut-sdp2}
  \begin{array}[!h]{rll}
    \text{Maximize}   & \iprod{\tfrac{1}{4}\Laplacian{G}(w)}{Y}          &                  \\[4pt]
    \text{subject to} & \iprod{\oprodsym{e_i}}{Y} = 1                    & \forall i \in V, \\[4pt]
                      & \iprod{\oprodsym{\ones}}{Y} \leq (\card{V}-2)^2, &                  \\[4pt]
                      & Y \in \Psd{V}.                                   &                  \\
  \end{array}
\end{equation}

Even this strengthened formulation is not exact for connected
bipartite graphs if we allow weights of arbitrary signs.  Consider,
for~instance, the path of length~\(3\) given by \(G =
([4],\set{12,23,34})\), with weights \(w = -\ones\).  Then MaxCut is
really a minimum cut problem and the optimal value is clearly \(-1\).
However, the feasible solution
\begin{equation*}
  \begin{bmatrix}
    1             & 1             & -\tfrac{1}{2} & -\tfrac{1}{2} \\[2.5pt]
    1             & 1             & -\tfrac{1}{2} & -\tfrac{1}{2} \\[2.5pt]
    -\tfrac{1}{2} & -\tfrac{1}{2} & 1             & 1             \\[2.5pt]
    -\tfrac{1}{2} & -\tfrac{1}{2} & 1             & 1             \\
  \end{bmatrix}
\end{equation*}
in~\eqref{eq:maxcut-sdp2} has objective value \(-3/4\).

These issues motivate the development of a theory of dual integrality
for weight functions in a cone, as~described in~\Cref{prob:KDI}.

\end{document}